\documentclass[a4paper,12pt]{article}
\usepackage{amsmath}
\usepackage{amssymb}
\usepackage{amsthm}
\usepackage{latexsym}
\usepackage{graphicx}
\usepackage{txfonts}
\usepackage{bm}
\usepackage{color}
\usepackage{lscape}

\usepackage{geometry}
\numberwithin{equation}{section}

\newtheorem{theorem}{Theorem}[section]
\newtheorem{proposition}[theorem]{Proposition}

\newtheorem{lemma}[theorem]{Lemma}
\newtheorem{remark}{Remark}[section]
\newtheorem{example}{Example}[section]

\newcommand{\OMIT}[1]{{\bf [OMIT:} #1 \ {\bf --- end OMIT] }}  
   \renewcommand{\OMIT}[1]{}            

\newcommand{\RR}{{\mathbb{R}}}
\newcommand{\ZZ}{{\mathbb{Z}}}
\newcommand{\vecone}{{\bf 1}}
\newcommand{\veczero}{{\bf 0}}
\newcommand{\dom}{{\rm dom\,}}

\newcommand{\suppp}{{\rm supp}\sp{+}}
\newcommand{\suppm}{{\rm supp}\sp{-}}

\newcommand{\unitvec}[1]{e\sp{#1}}

\newcommand{\argmax}{\arg \max}
\newcommand{\argmin}{\arg \min}

\newcommand{\finbox}{\hspace*{\fill}$\rule{0.2cm}{0.2cm}$}

\newcommand{\Lnat}{{L$^{\natural}$}}
\newcommand{\Mnat}{{M$^{\natural}$}}
\newcommand{\LLnat}{{L$^{\natural}_{2}$}}
\newcommand{\MMnat}{{M$^{\natural}_{2}$}}
\newcommand{\LL}{{L$_{2}$}}
\newcommand{\MM}{{M$_{2}$}}

\newcommand{\Rinf}{\RR \cup \{ +\infty\}}

\newcommand{\Bvexb}{\mbox{\rm\bf (B-EXC)}}
\newcommand{\Bnvex}{\mbox{\rm (B$\sp{\natural}$-EXC)} }
\newcommand{\Bnvexb}{\mbox{\rm\bf (B$\sp{\natural}$-EXC)}}

\newcommand{\Mvexb}{\mbox{\rm\bf (M-EXC)}}
\newcommand{\Mnvex}{\mbox{\rm (M$\sp{\natural}$-EXC)} }
\newcommand{\Mnvexb}{\mbox{\rm\bf (M$\sp{\natural}$-EXC)}}

\begin{document}

\title{Inclusion and Intersection Relations Between
Fundamental Classes of Discrete Convex Functions%
}

\author{
Satoko Moriguchi%
\thanks{
Faculty of Economics and Business Administration,
Tokyo Metropolitan University, 
satoko5@tmu.ac.jp}
\ and 
Kazuo Murota%
\thanks{
The Institute of Statistical Mathematics,
and Faculty of Economics and Business Administration,
Tokyo Metropolitan University, 
murota@tmu.ac.jp}
}

\date{November 2021 / July 2022 / February 2023}

\maketitle

\begin{abstract}

In discrete convex analysis,
various convexity concepts are considered for discrete functions
such as separable convexity, L-convexity,
M-convexity, integral convexity, and multimodularity.
These concepts of discrete convex functions are not mutually independent.
For example, \Mnat-convexity is a special case of integral convexity,
and the combination of \Lnat-convexity and \Mnat-convexity
coincides with separable convexity.
This paper aims at a fairly comprehensive analysis
of the inclusion and intersection relations 
for various classes of discrete convex functions.
Emphasis is put on the analysis of multimodularity
in relation to \Lnat-convexity and \Mnat-convexity.
\end{abstract}

{\rm L}-convex function, 
{\rm M}-convex function, 
Multimodular function,
Separable convex function






\section{Introduction}
\label{SCintro}

A function defined on  the integer lattice $\ZZ\sp{n}$ is called 
a discrete function.
Various types of convexity   have been defined for discrete functions
in discrete convex analysis \cite{Fuj05book,Mdca98,Mdcasiam,Mdcaprimer07,Mbonn09},
such as 
L-convex functions, \Lnat-convex functions,
M-convex functions, \Mnat-convex functions, integrally convex functions,
and multimodular functions.
\Lnat-convex functions have applications in several fields including
operations research (inventory theory, scheduling, etc.)
\cite{Che17,CL21dca,SCB14},
economics and auction theory \cite{Mdcaeco16,Shi17L},
and computer vision \cite{Shi17L}.
\Mnat-convex functions have applications in
operations research \cite{CL21dca}
and economics \cite{Mdcasiam,Mbonn09,Mdcaeco16,MTcompeq03,ST15jorsj}.
Multimodular functions are used for queueing, 
Markov decision processes, and  discrete-event systems 
\cite{AGH00,AGH03,Haj85}.

The classes of discrete convex functions are not mutually independent.
In some cases, two classes of discrete convex functions
have an inclusion relation
as an (almost) immediate consequence of the definitions.
This is the case, for example, with the inclusion of 
M-convex functions in the set of \Mnat-convex functions.
In other cases, inclusion relations are established as theorems.
For example, it is a theorem \cite[Theorem~6.42]{Mdcasiam} that
\Mnat-convex functions are integrally convex functions.
This paper aims at a fairly comprehensive analysis
of the inclusion relations for various classes of discrete convex functions.

In this paper we are also interested in 
what is implied by the combination of two convexity properties.
For example, it is known 
(\cite[Theorem~3.17]{MS01rel}, \cite[Theorem~8.49]{Mdcasiam})
that a function is both \Lnat-convex and \Mnat-convex 
if and only if it is separable convex, 
where a function 
$f:\ZZ\sp{n}\to \Rinf$
is called
{\em separable convex} if it is represented as
\begin{equation} \label{sepcnvdef}
f(x) = \sum_{i=1}\sp{n} \varphi_{i}(x_{i}) \qquad (x\in \ZZ\sp{n})
\end{equation}
with univariate functions
$\varphi_{i}$ $(i=1,2,\ldots ,n)$
that satisfy
$\varphi_{i}(t-1) + \varphi_{i}(t+1)\geq 2 \varphi_{i}(t)$ for all $t\in \ZZ$
and are finite-valued on intervals of integers. 
It is easy to see that a separable convex function 
is both \Lnat-convex and \Mnat-convex, and hence
the content of the above statement lies in the 
claim that the combination of \Lnat-convexity and \Mnat-convexity
implies separable convexity.
In this paper we are interested in such statements
for other classes of discrete convex functions.
In particular, we show in Section~\ref{SCl2natmm} 
that a function is both \Lnat-convex and multimodular
if and only if it is separable convex.

It is noted, however, that such relationship may not be so simple
for other pairs of function classes.
As an example, consider \Mnat -convex functions and multimodular functions.
It is known \cite{MM19multm}
that \Mnat-convex functions and multimodular functions are the same
for functions in two variables.
For functions in more variables, 
this is not the case,
and there is no inclusion relation between the classes of \Mnat-convex functions 
(resp., sets)
and that of multimodular functions
(resp., sets).
Figure~\ref{FGmnatmm} shows examples of integral polymatroids
(\Mnat-convex sets) with and without multimodularity 
(see Examples \ref{EXmnatmmdim3} and \ref{EXmnatNOTmmdim3} for details).
The indicator function of the set in (a) is both \Mnat-convex and multimodular,
but not separable convex.
A systematic study in Section~\ref{SCmnatmm}
will show that there are infinitely many functions (resp., sets) that are 
both \Mnat-convex and multimodular, but not separable convex.

\begin{figure}
\begin{center}
\includegraphics[height=50mm]{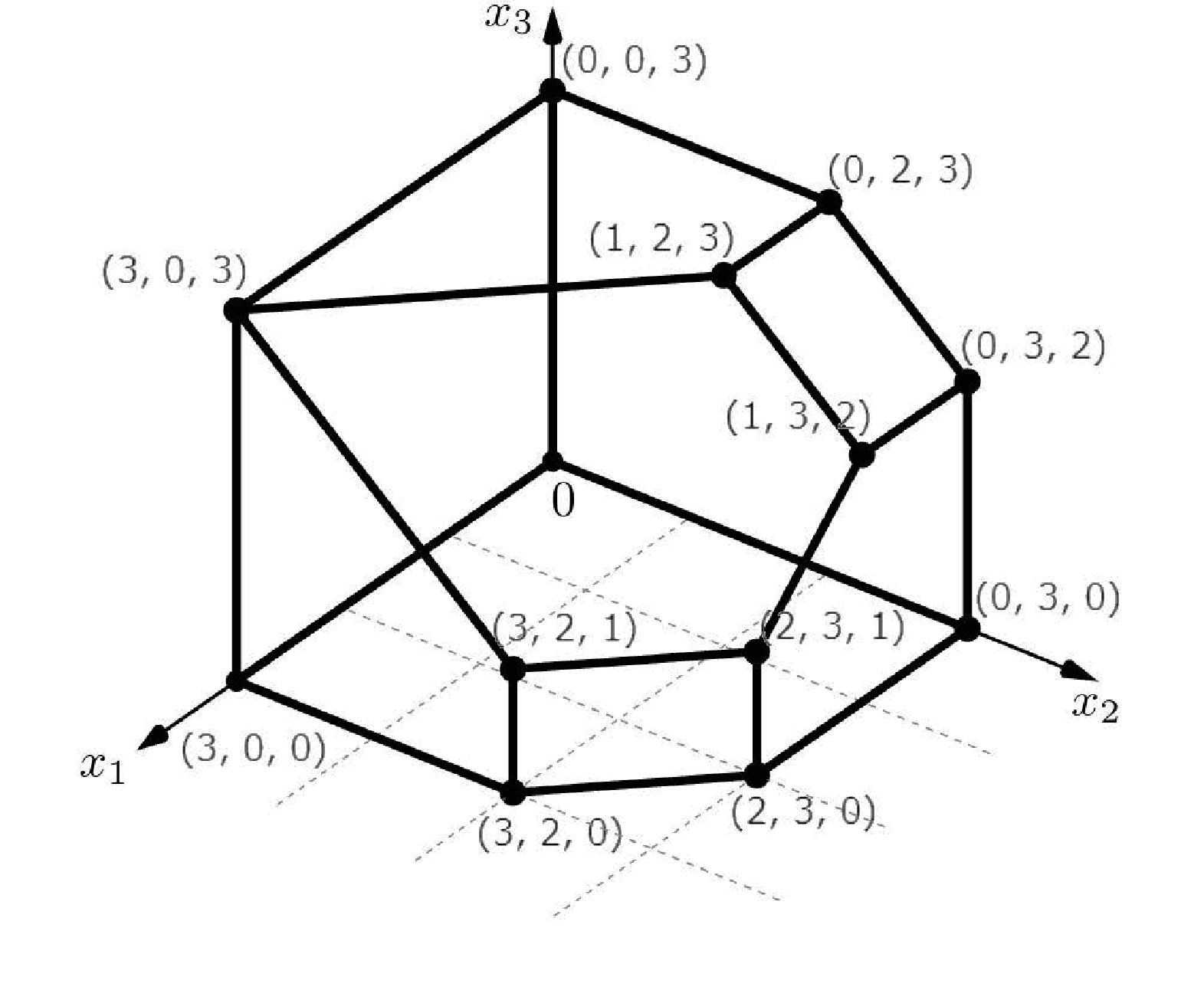}
\qquad
\includegraphics[height=50mm]{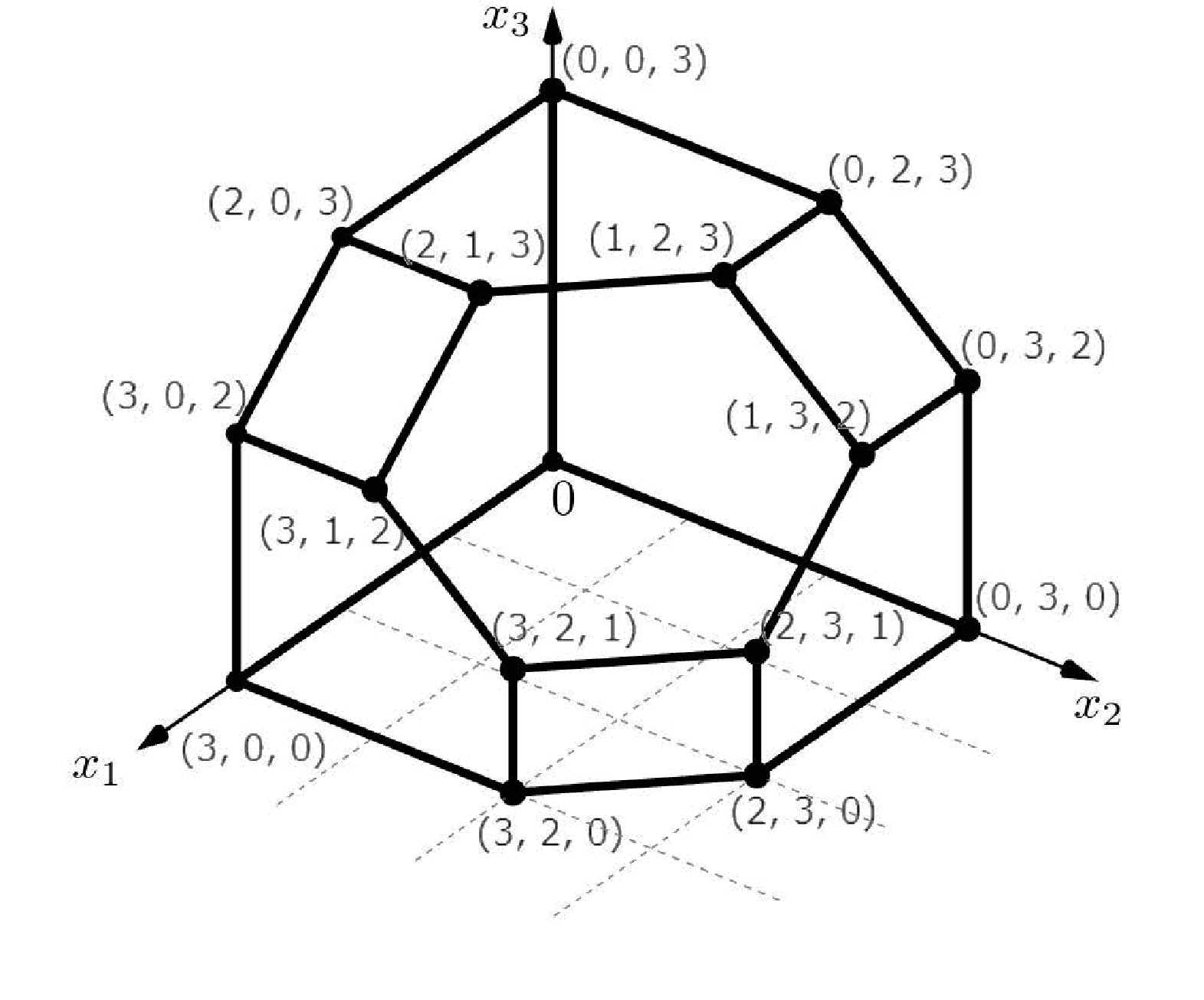}
\\
(a) \Mnat-convex and multimodular
\qquad
(b) \Mnat-convex and not multimodular
\end{center}
\caption{\Mnat-convex sets (polymatroids) with and without multimodularity}
\label{FGmnatmm}
\end{figure}

This paper is organized as follows.
Section~\ref{SCprelim} introduces notations and
summarizes the known pairwise inclusion relations between classes of 
discrete convex functions
and shows how the analysis for functions can be reduced to that for sets.
The main results are given in 
Sections \ref{SCinclude} and \ref{SCinter}
with illustrative examples.
Section~\ref{SCresult} presents technical contributions, containing 
theorems and proofs about multimodularity.
Definitions of various concepts of discrete convex functions
are given in Appendix for the convenience of reference.

\section{Preliminaries}
\label{SCprelim}

\subsection{Notation}
\label{SCnotat}

Let $n$ be a positive integer and 
$N = \{ 1,2,\ldots, n  \}$.
For a subset $A$ of $N$,
we denote 
by $\unitvec{A}$ 
the characteristic vector of $A$;
the $i$th component of $\unitvec{A}$ is equal to 1 or 0 according 
to whether $i \in A$ or not.
We use a short-hand notation $\unitvec{i}$ for $\unitvec{ \{ i \} }$,
which is the $i$th unit vector. 
The vector with all components equal to 1 is denoted by $\vecone$,
that is, $\vecone =(1,1,\ldots, 1) = \unitvec{N}$.
We sometimes use $\unitvec{0}$ to mean the zero vector $\veczero$.

For a vector $x = (x_{1}, x_{2}, \ldots, x_{n})$
and a subset $A$ of $N$,
we use notation $x(A) = \sum_{i \in A} x_{i}$.
For two vectors $x, y \in \RR\sp{n}$,
the vectors of componentwise maximum and minimum of $x$ and $y$
are denoted, respectively, 
by $x \vee y$ and $x \wedge y$,
i.e.,
\begin{equation} \label{veewedgedef}
  (x \vee y)_{i} = \max(x_{i}, y_{i}),
\quad
  (x \wedge y)_{i} = \min(x_{i}, y_{i})
\qquad (i =1,2,\ldots, n).
\end{equation}
For a real number $t \in \RR$, 
$\left\lceil  t   \right\rceil$ 
denotes the smallest integer not smaller than $t$
(rounding-up to the nearest integer)
and $\left\lfloor  t  \right\rfloor$
the largest integer not larger than $t$
(rounding-down to the nearest integer),
and this operation is extended to a vector
by componentwise applications.
That is,
\begin{equation}  \label{ceilfloor}
\left\lceil  x   \right\rceil =
(\lceil x_{1} \rceil, \lceil x_{2} \rceil, \ldots, \lceil x_{n} \rceil),
\qquad
\left\lfloor  x   \right\rfloor =
(\lfloor x_{1} \rfloor, \lfloor x_{2} \rfloor, \ldots, \lfloor x_{n} \rfloor).
\end{equation}
For integer vectors 
$a \in (\ZZ \cup \{ -\infty \})\sp{n}$ and 
$b \in (\ZZ \cup \{ +\infty \})\sp{n}$ 
with $a \leq b$, 
the box of reals and the box of integers 
(discrete rectangle, integer interval)
between $a$ and $b$ 
are denoted by $[a,b]_{\RR}$ and $[a,b]_{\ZZ}$, respectively, i.e.,
\begin{align} 
[a,b]_{\RR} & = \{ x \in \RR\sp{n} \mid a_{i} \leq x_{i} \leq b_{i} \ (i=1,2,\ldots,n)  \},
\label{boxdefR} 
\\
[a,b]_{\ZZ} &= \{ x \in \ZZ\sp{n} \mid a_{i} \leq x_{i} \leq b_{i} \ (i=1,2,\ldots,n)  \}.
\label{boxdefZ} 
\end{align}

We consider functions defined on integer lattice points, 
$f: \ZZ\sp{n} \to \RR \cup \{ +\infty \}$,
where the function may possibly take $+\infty$.
The effective domain of $f$ means the set of $x$
with $f(x) <  +\infty$ and is denoted as 
\begin{equation}  \label{domfdef}
\dom f =   \{ x \in \ZZ\sp{n} \mid  f(x) < +\infty \}.
\end{equation}
We always assume that $\dom f$ is nonempty.
The set of the minimizers of $f$ is denoted by $\argmin f$. 
For a function $f$ and a vector $p$, 
$f[-p]$ denotes the function defined by
\begin{equation}  \label{fpdef}
 f[-p](x) = f(x) - \sum_{i=1}\sp{n} p_{i} x_{i}.
\end{equation}

The convex hull of a set $S$ $(\subseteq \ZZ\sp{n})$ is denoted by $\overline{S}$.
The indicator function of a set $S \subseteq \ZZ\sp{n}$
is denoted by $\delta_S$, which is the function 
$\delta_{S}: \ZZ\sp{n} \to \{ 0, +\infty \}$
defined by
\begin{equation}  \label{indicatordef}
\delta_{S}(x)  =
   \left\{  \begin{array}{ll}
    0            &   (x \in S) ,      \\
   + \infty      &   (x \not\in S) . \\
                      \end{array}  \right.
\end{equation}

\subsection{Classes of discrete convex functions}
\label{SCclassdcf}

In this paper we investigate  
pairwise relations for the following 13 classes of discrete convex functions:
separable convex functions,
integrally convex functions,
L-convex functions,  
\Lnat-convex functions,  
\LL-convex functions, 
\LLnat-convex functions, 
M-convex functions,  
\Mnat-convex functions,  
\MM-convex functions, 
\MMnat-convex functions, 
multimodular functions,
globally discrete midpoint convex functions,
and directed discrete midpoint convex functions.
The separable convex functions are defined by \eqref{sepcnvdef}
in Introduction.
The definitions of other functions 
are given in Appendix \ref{SCdiscfndef},
except that the definition of multimodular functions 
is given in Section~\ref{SCmmfnDef}.

The inclusion relations between these function classes
are summarized in the following proposition,
with Remark \ref{RMfnclassbib} giving the references.
The inclusion relations \eqref{fnfamL}--\eqref{fnfamM2}
are depicted in Fig.~\ref{FGclassVenn}.

\begin{proposition} \label{PRfnclassInc}
The following inclusion relations hold for functions on $\ZZ\sp{n}$:
\begin{align}
&
\{ \mbox{\rm separable conv.} \}  \subsetneqq \ 
\{ \mbox{\rm \Lnat-conv.} \} \subsetneqq \ \{ \mbox{\rm integrally conv.} \},
\label{fnfamL}
\\ &
\{ \mbox{\rm separable conv.} \}  \subsetneqq \ 
\{ \mbox{\rm \Mnat-conv.} \}
\subsetneqq \ \{ \mbox{\rm integrally conv.} \},
\label{fnfamM}
\\ &
\{ \mbox{\rm L-conv.} \}  \subsetneqq \
  \left\{  \begin{array}{l}  \{ \mbox{\rm \Lnat-conv.} \}
                          \\ \{ \mbox{\rm \LL-conv.} \} \end{array}  \right \}
\subsetneqq \  
\{ \mbox{\rm \LLnat-conv.} \} 
\subsetneqq  \ \{ \mbox{\rm integrally conv.} \} ,
\label{fnfamL2}
\\ &
\{ \mbox{\rm M-conv.} \}  \subsetneqq \
  \left\{  \begin{array}{l}  \{ \mbox{\rm \Mnat-conv.} \}
                          \\ \{ \mbox{\rm \MM-conv.} \} \end{array}  \right \}
\subsetneqq \  
\{ \mbox{\rm \MMnat-conv.} \} 
\subsetneqq  \ \{ \mbox{\rm integrally conv.} \},
\label{fnfamM2}
\\ &
\{ \mbox{\rm separable conv.} \}
\subsetneqq \ \{ \mbox{\rm multimodular} \}
\subsetneqq \ \{ \mbox{\rm integrally conv.} \} ,
\label{fnfammultm}
\\ &
\{ \mbox{\rm \Lnat-conv.} \}
\subsetneqq \ \{ \mbox{\rm globally d.m.c.} \}
\subsetneqq \ \{ \mbox{\rm integrally conv.} \},
\label{fnfamdmc}
\\ & 
\{ \mbox{\rm \Lnat-conv.} \}
\subsetneqq \ \{ \mbox{\rm directed d.m.c.} \}
\subsetneqq \ \{ \mbox{\rm integrally conv.} \}.
\label{fnfamddmc}
\end{align}
\vspace{-1.3\baselineskip}\\
\finbox
\end{proposition}

\begin{figure}\begin{center}
\includegraphics[height=70mm]{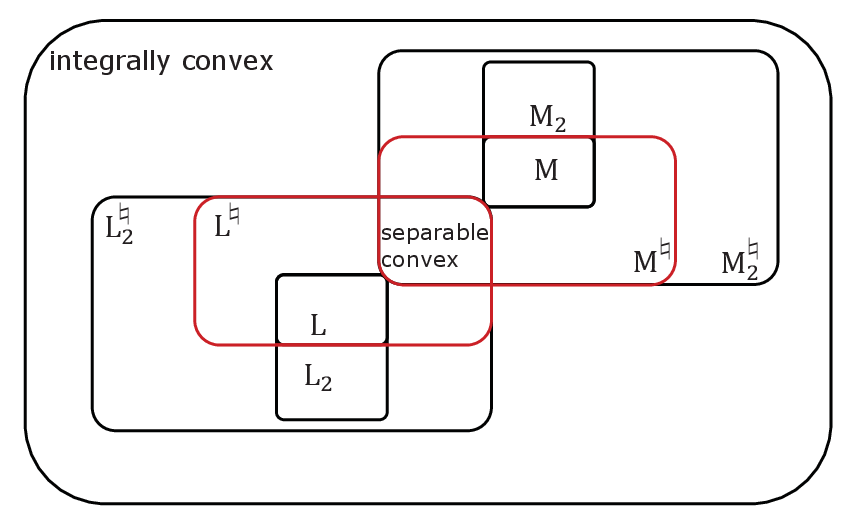} 
\caption{Inclusion relations between classes of discrete convex functions}
\label{FGclassVenn}
\end{center}\end{figure}

\begin{remark} \rm \label{RMfnclassbib}
Here is a supplement to Proposition~\ref{PRfnclassInc}.
Integral convexity is established for 
\Lnat-convex functions in 
\cite[Theorem~7.20]{Mdcasiam},
for \Mnat-convex functions in 
\cite[Theorem~3.9]{MS01rel}
(see also \cite[Theorem~6.42]{Mdcasiam}), and
for globally discrete midpoint convex functions in \cite[Theorem~6]{MMTT20dmc}.
The inclusion relations 
\eqref{fnfamL2} and \eqref{fnfamM2}
concerning \LL-convex and \MM-convex functions
are given in \cite[Section~8.3]{Mdcasiam}.
The integral convexity of multimodular functions in \eqref{fnfammultm} was
pointed out first in \cite[Section~14.6]{Mdcaprimer07},
while this is implicit in the construction of the convex extension given earlier in 
\cite[Theorem~4.3]{Haj85}.
The first inclusion in \eqref{fnfammultm} for multimodular functions 
is given in \cite[Proposition~2]{MM19multm}.
The inclusion relations in \eqref{fnfamdmc} are given in \cite[Theorem~6]{MMTT20dmc},
and those in \eqref{fnfamddmc} are in \cite{TT21ddmc}.
\finbox
\end{remark}

\subsection{Our approach to identify the intersection of two classes}
\label{SCreducefnset}

When two function classes have no inclusion relation between themselves,
we are interested in 
identifying the intersection of these classes.
A typical (known) statement of this kind is that 
a function is both \Lnat-convex and \Mnat-convex 
if and only if it is separable convex.
In this paper (Section~\ref{SCinter}) we are concerned with similar statements:
\begin{equation} \label{ABCfun}
\mbox{A function is both A-convex and B-convex if and only if it is C-convex},
\end{equation}
where 
A-, B-, and C-convex 
denote certain specified discrete convexity
from among those treated in this paper.

In discrete convex analysis, it is generally true that
a set $S$  \ ($\subseteq \ZZ\sp{n}$) 
has a certain discrete convexity 
if and only if 
its indicator function $\delta_S$ 
has the same discrete convexity.
Therefore, the statement \eqref{ABCfun}
implies the corresponding statement for a set:
\begin{equation} \label{ABCset}
\mbox{A set is both A-convex and B-convex if and only if it is C-convex}.
\end{equation}
While the implication ``\eqref{ABCfun} $\Rightarrow$ \eqref{ABCset}''
is obvious, the converse is also true 
under mild assumptions.
In such cases, we can reduce the proof of \eqref{ABCfun} for functions 
to that of \eqref{ABCset} for sets.

As a technical assumption
for the reverse implication 
 ``\eqref{ABCset} $\Rightarrow$ \eqref{ABCfun}''
we introduce the condition 
on a function $f:\ZZ\sp{n} \to \Rinf$ that
\begin{equation} \label{ABCfunAssm}
\mbox{(i) $\dom f$ is bounded or 
(ii) $\overline{\dom f}$ is a polyhedron and $f$ is convex-extensible},
\end{equation}
where, in this paper, we call a function $f$  
{\em convex-extensible}
if it is extensible to a locally polyhedral convex function
in the sense of \cite[Section~15]{Fuj05book}.
For the set version,
we consider the condition on a set $S \subseteq \ZZ\sp{n}$ that 
\begin{equation} \label{ABCsetAssm}
\mbox{(i) $S$ is bounded or 
(ii) $\overline{S}$ is a polyhedron and $S = \overline{S} \cap \ZZ\sp{n}$}.
\end{equation}

For the sake of exposition, let us 
choose separable convexity as ``C-convexity'' in the above generic statement.
Note that a box of integers is 
the set version of a separable convex function.
We use the following fact, 
where
$\argmin f[-p]$ means the set of minimizers of the function
$f[-p]$ defined in \eqref{fpdef}.

\begin{proposition}\label{PRscf-fp}
Under the assumption \eqref{ABCfunAssm},
a function $f:\ZZ\sp{n} \to \Rinf$ is separable convex 
if and only if,
for any vector $p \in \RR\sp{n}$, $\argmin f[-p]$ is a box of integers
or an empty set.
\end{proposition}
\begin{proof}
See Remark~\ref{RMscf-fp} at the end of Section~\ref{SCreducefnset}.
\end{proof}

For ``A-convexity'' and ``B-convexity'' in the generic statement,
we assume the following natural properties:
\begin{align}
& \mbox{%
If a function $f$ is A-convex (resp., B-convex), then $f$ satisfies \eqref{ABCfunAssm}
}
\nonumber \\ & 
\mbox{and, for any $p \in \RR\sp{n}$, $\argmin f[-p]$ is A-convex  (resp., B-convex)}.
\label{cnvABargminfp}
\end{align}
By Proposition~\ref{PRscf-fp} (only-if part), separable convexity has this property.
Moreover, it is known
that all kinds of discrete convexity treated 
in this paper have this property.

\begin{proposition}\label{PRscf-box}
Under the assumptions \eqref{ABCfunAssm}--\eqref{cnvABargminfp},
the following two statements are equivalent:
\\
{\rm (1)} 
A set which is both A-convex and B-convex is a box of integers.
\\
{\rm (2)} 
A function which is both A-convex and B-convex is a separable convex function.
\end{proposition}

\begin{proof}
The implication (2) $\Longrightarrow$ (1) is obvious, as 
(1) for a set $S$ follows from  (2) for its indicator function $\delta_{S}$.
The converse, (1) $\Longrightarrow$ (2), can be shown as follows.
Let $f$ be both A-convex and B-convex.
By the assumption \eqref{cnvABargminfp},
the set $S = \argmin f[-p]$ is both A-convex and B-convex
for each $p \in \RR\sp{n}$.
Then, by (1),
$\argmin f[-p]$ is a box for each $p \in \RR\sp{n}$.
It follows from this and Proposition~\ref{PRscf-fp} that $f$ is separable convex.
\end{proof}

Similarly, we can obtain the following propositions,
where ``C-convexity'' in the generic statement
is replaced by {\rm L}-convexity and {\rm M}-convexity, respectively.

\begin{proposition}\label{PRlf-lset}
Under the assumptions \eqref{ABCfunAssm}--\eqref{cnvABargminfp},
the following two statements are equivalent:
\\
{\rm (1)} 
A set which is both A-convex and B-convex is an L-convex set.
\\
{\rm (2)} 
A function which is both A-convex and B-convex is an L-convex function. 
\end{proposition}
\begin{proof}
The proof is the same as that of Proposition~\ref{PRscf-box},
except that Proposition~\ref{PRscf-fp} is replaced by Theorem~\ref{THlf-fp}.
\end{proof}

\begin{proposition}\label{PRmf-mset}
Under the assumptions \eqref{ABCfunAssm}--\eqref{cnvABargminfp},
the following two statements are equivalent:
\\
{\rm (1)} 
A set which is both A-convex and B-convex is an M-convex set.
\\
{\rm (2)} 
A function which is both A-convex and B-convex is an M-convex function.
\end{proposition}
\begin{proof}
The proof is the same as that of Proposition~\ref{PRscf-box},
except that Proposition~\ref{PRscf-fp} is replaced by Theorem~\ref{THmf-fp}.
\end{proof}

Proposition~\ref{PRscf-box} 
will be used in the proof of Theorem~\ref{THl2natmmB},
and 
Propositions \ref{PRlf-lset} and \ref{PRmf-mset}
will be used in the proofs of Theorems \ref{THl2lnatL} and \ref{THm2mnatM}, 
respectively.

\begin{remark} \rm \label{RMscf-fp}
A proof of Proposition~\ref{PRscf-fp} is outlined here for completeness.
Suppose that $f$ is a separable convex function represented as
$f(x) = \sum_{i=1}\sp{n} \varphi_{i}(x_{i})$
in \eqref{sepcnvdef}. 
Then
$f[-p](x) = \sum_{i=1}\sp{n} \varphi_{i}[-p_{i}](x_{i})$,
from which $x$ is a minimizer of $f[-p]$ 
if and only if, 
for each $i$, 
$x_{i}$ is a minimizer of $\varphi_{i}[-p_{i}]$. 
Therefore
$\argmin f[-p]$ is a box $[a, b]_{\ZZ}$,
where $a$ and $b$ are the vectors whose components are defined by
$[a_{i}, b_{i}]_{\ZZ} = \argmin \varphi_{i}[-p_{i}]$
$(i=1,2,\ldots, n)$.
To show the converse under \eqref{ABCfunAssm}, suppose that
$\argmin f[-p]$ is a box of integers for any $p \in \RR\sp{n}$. 

First we consider the case (ii) where
$\overline{\dom f}$ is a polyhedron and $f$ is convex-extensible.
Let $\overline{f}$ denote the convex extension of $f$.
We use notation 
$S_{p} := \argmin f[-p]$.
By the assumption, 
$S_{p}$ is a box of integers 
and 
$\overline{S_{p}} = \overline{\argmin f[-p]} =  \argmin \overline{f}[-p]$
is a box of reals.
We have
$\overline{\dom f} = \bigcup_{p} \overline{S_{p}}$,
which implies that 
$\overline{\dom f}$ is a box of reals and hence
$\dom f$ is a box of integers.
Fix an arbitrary $z \in \dom f$.
For $i=1,2,\ldots,n$, define 
$\varphi_{i}: \ZZ \to \Rinf$
by
$\varphi_{i}(t) = f(z_{1}, \ldots, z_{i-1}, t, z_{i+1}, \ldots, z_{n})$
for $t \in \ZZ$.
By the assumed convex-extensibility of $f$, we have that
$\dom \varphi_{i}$ is an interval of integers and 
$\varphi_{i}(t-1) + \varphi_{i}(t+1)\geq 2 \varphi_{i}(t)$ for all $t\in \ZZ$.
Using these univariate functions we obtain the desired expression
$f(x) = \sum_{i=1}\sp{n} \varphi_{i}(x_{i})$.

Next, we turn to the other case (i) where $\dom f$ is assumed to be bounded.
Let $\overline{f}$ denote the convex envelope of $f$,
by which we mean that 
$\overline{f}(x)$ is equal to the maximum of 
$g(x)$ taken over all closed convex functions $g$ 
satisfying $g(y) \leq f(y)$ for all $y \in \dom f$.
Since $\dom f$ is bounded,
$\overline{f}$ is a polyhedral convex function
and 
$\dom \overline{f} = \overline{\dom f}$
is a polyhedron.
To show $\overline{f}=f$ (convex-extensibility of $f$) by contradiction,
suppose that $\overline{f}(y) < f(y)$ for some $y \in \dom f$.
Let $p$ be a subgradient of $\overline{f}$ at $y$,
which is equivalent to saying that
$y \in \argmin \overline{f}[-p]$.
Since
$\argmin \overline{f}[-p] = \overline{\argmin f[-p]}$,
this implies 
$y \in \overline{\argmin f[-p]}$,
whereas
$y \notin \argmin f[-p]$ by $\overline{f}(y) < f(y)$.
This is a contradiction to the assumption that 
$\argmin f[-p]$ is a box of integers.
Hence $f$ is convex-extensible, and 
the proof in the case of (i) is reduced to that in the case of (ii).
\finbox
\end{remark}

\section{Inclusion Relations Between Convexity Classes}
\label{SCinclude}

In Section~\ref{SCclassdcf}
we have presented basic inclusion relations
for fundamental classes of discrete convex functions
(Proposition~\ref{PRfnclassInc} and Fig.~\ref{FGclassVenn}).
The objective of this section is to offer  supplementary facts.
In particular, we present examples to demonstrate
proper inclusion ($\subsetneqq$),
and cover (globally, directed) discrete midpoint convex functions.
The relation of multimodularity to \Lnat- and \Mnat-convexity
will be investigated later in Sections \ref{SCl2natmm} and \ref{SCmnatmm}.

\begin{table} 
\begin{center}
\caption{Relations between classes of discrete convex functions}
\label{TBinclinter}
\medskip
\addtolength{\tabcolsep}{-2pt}

\centering  
\small
    \begin{tabular}{|l|l|l||l|l|l|l||l|l|l|l||l|l|l|l|}
    \hline
& sep & int-c  & L & \LL  & \Lnat  & \LLnat  & M & \MM  & \Mnat  & \MMnat  & m.m. & g-dmc & d-dmc 
\\ \hline
sep & = & $\subset$ & $\triangledown$ & $\triangledown$ & $\subset$ & $\subset$ 
   & $\triangledown$ & $\triangledown$ & $\subset$ & $\subset$ 
   & $\subset$ $^{\circ}$ & $\subset$  & $\subset$ 
\\
 &  &  & lin & lin &  &  & point & point &  &  &  &   &  
\\ \hline
int-c &  & = & $\supset$ $^{\circ}$ & $\supset$ $^{\circ}$ & $\supset$ $^{\circ}$ & $\supset$ $^{\circ}$ 
      & $\supset$ $^{\circ}$  & $\supset$ $^{\circ}$ & $\supset$ $^{\circ}$ & $\supset$ $^{\circ}$ 
      & $\supset$ $^{\circ}$ & $\supset$ $^{\circ}$ & $\supset$ $^{\circ}$ 
 \\
  &  &  &  &  &  &  &  &  &  &  &  &   &  
\\ \hline\hline
L &  &  & = & $\subset$ & $\subset$ & $\subset$ & $\triangledown$ & $\triangledown$ 
   & $\triangledown$ & $\triangledown$ & $\triangledown$ & $\subset$  & $\subset$ 
\\
 &  &  &  & &  &  & none & none & lin & lin & lin &   &  
\\ \hline
\LL  &  &  &  & = & $\triangledown$ $^*$  & $\subset$ 
   & $\triangledown$ & $\triangledown$  & $\triangledown$ & $\triangledown$ 
   & $\triangledown$ & $\triangledown$ $^*$  & $\triangledown$  
\\
  &  &  &  &  & L &  & none & none & lin & lin & lin & $\supseteq$ L & $\supseteq$ L 
\\ \hline
\Lnat  &  &  &  &  & = & $\subset$ 
   & $\triangledown$ & $\triangledown$  & $\triangle$ $^{\circ}$  & $\triangle$ 
   & $\triangle$ $^*$ & $\subset$ $^{\circ}$ & $\subset$ $^{\circ}$ 
\\
  &  &  &  &  &  &  & point  & point  & sep & sep & sep &   &  
\\ \hline
\LLnat  &  &  &  &  &  & = 
   & $\triangledown$ & $\triangledown$  & $\triangle$ & $\triangle$ $^{\circ}$ 
   & $\triangle$ $^*$ &  $\triangle$ $^*$  &  $\triangle$ $^{\circ}$ 
\\
  &  &  &  &  &  &  & point & point  & sep & sep & sep &  $\supseteq$ \Lnat & $\supseteq$ \Lnat
\\ \hline\hline
M &  &  &  &  &  &  & = & $\subset$ & $\subset$ & $\subset$ & 
$\triangledown$ $^*$ & $\triangledown$ $^*$  & $\triangledown$ $^*$ 
\\
&  &  &  &  &  &  &  &  &  &  &  &  & 
\\ \hline
\MM  &  &  &  &  &  &  &  & = & $\triangledown$ $^*$ & $\subset$ & $\triangledown$ $^*$ 
   & $\triangledown$ $^*$  & $\triangledown$ $^*$ 
\\
  &  &  &  &  &  &  &  &  & M &  &  &  &  
\\ \hline
\Mnat  &  &  &  &  &  &  &  &  & = & $\subset$ & $\triangle$ $^*$ & $\triangle$ $^*$  
  & $\triangle$ $^*$ 
\\
  &  &  &  &  &  &  &  &  &  &  &  $\supsetneqq$ sep & $\supsetneqq$ sep & $\supsetneqq$ sep
\\ \hline
\MMnat  &  &  &  &  &  &  &  &  &  & = & $\triangle$ $^*$ & $\triangle$ $^*$  & $\triangle$ $^*$ 
\\
  &  &  &  &  &  &  &  &  &  &  &  $\supsetneqq$ sep  & $\supsetneqq$ sep & $\supsetneqq$ sep 
\\ \hline\hline
m.m. &  &  &  &  &  &  &  &  &  &  & = & $\triangle$ $^*$  & $\triangle$ $^*$ 
\\
 &  &  &  &  &  &  &  &  &  &  &  & $\supsetneqq$ sep & $\supsetneqq$ sep
\\ \hline
g-dmc &  &  &  &  &  &  &  &  &  &  &  & =  &$\triangle$   
\\
 &  &  &  &  &  &  &  &  &  &  &  &   &  $\supsetneqq$ \Lnat  
\\ \hline
d-dmc &  &  &  &  &  &  &  &  &  &  &  &   & = 
 \\
 &  &  &  &  &  &  &  &  &  &  &  &   & 
\\ \hline
\end{tabular}
\addtolength{\tabcolsep}{2pt}
\end{center}
sep = separable convex, \ 
int-c = integrally convex, \
m.m. = multimodular. \
\\
g-dmc=globally discrete midpoint convex, \ 
d-dmc=directed discrete midpoint convex.
\\
$\subset$: 
The class of the row is properly contained ($\subset$) by the class of the column.
\\
$\supset$: 
The class of the row properly contains ($\supset$) the class of the column.
\\
$\triangle$: 
There is no inclusion relation between the classes of the row and the column,
and their intersection includes all separable convex functions.
\\
$\triangledown$: 
There is no inclusion relation between the classes of the row and the column,
and their intersection does 
not include the set of separable convex functions.
\\
lin: linear function on $\ZZ\sp{n}$,  \ 
point: function defined on a single point. 
\\
$^{*}$: results of this paper, 
$^{\circ}$: results in previous studies,
Unmarked: easy observations.
\end{table}

Table~\ref{TBinclinter} is a summary of the inclusion and intersection relations 
for various classes of discrete convex functions.
In each row, corresponding to a class of discrete convex functions, 
the first (upper) line shows its inclusion relation 
by $\subset$, $\supset$, $\triangle$, and $\triangledown$.
The symbol $\subset$ (resp., $\supset$) means that 
the class of the row is properly contained in
(resp., properly contains) the class of the column.
The symbols $\triangle$ and $\triangledown$ indicate that 
there is no inclusion relation between the classes of the row and the column;
$\triangle$ or $\triangledown$ is used depending on 
whether their intersection includes all separable convex functions
or not.
The second (lower) line of each row shows what their intersection is,
to be treated in Section~\ref{SCinter}.

A set $S$ is always assumed to be a subset of $\ZZ\sp{n}$ and a function $f$ 
is defined on $\ZZ\sp{n}$, that is, $f:\ZZ\sp{n}\to \Rinf$.
We can demonstrate that a certain class of discrete convex functions,
say, A-convex functions
is not contained in another class of discrete convex functions,
say, B-convex functions
by exhibiting a set $S$ 
that is A-convex and not B-convex.
Note that, in this case, 
the indicator function $\delta_{S}$
is an A-convex function which is not B-convex.

\subsection{L-convexity, M-convexity, and multimodularity}
\label{SCnorelLMmm}

The following examples are concerned with 
L-convexity,  M-convexity, and multimodularity (m.m.).
For function classes $A, B, C$ with $\{ A, B, C \} = \{ \mbox{L}, \mbox{M}, \mbox{m.m.} \}$,
we have $A \setminus (B \cup C) \ne \emptyset$.%

\begin{example}  \rm  \label{EXnorelLnotMmm}
Let
$S = \{ x \in \ZZ\sp{3}  \mid x= t (1,1,1), t \in \ZZ \}$.
This set is L-convex (hence \Lnat-convex),
but it is neither \Mnat-convex nor multimodular.
Multimodularity of $S$
is denied in Example~\ref{EXnorelLnotMmm-2} in Section \ref{SCmmfnDef}.
\finbox
\end{example}

\begin{example}  \rm  \label{EXnorelmmnotLM}
Let
$S = \{ x \in \ZZ\sp{3}  \mid x= t (1,-1,1), t \in \ZZ \}$.
This set is multimodular but it is neither \Lnat-convex nor \Mnat-convex.
Multimodularity of $S$ will be 
demonstrated in Example~\ref{EXnorelmmnotLM-2} in Section \ref{SCmmpolydesc}.
\finbox
\end{example}

\begin{example}  \rm  \label{EXnorelMnatnotLmm}
Let
\begin{equation} 
S = 
\{ x\in \ZZ\sp{3} \mid 0\leq x_i\leq 3~(i=1,2,3),
\ x(X) \leq 5 \ (|X|=2), \ x_{1}+x_{2}+x_{3} \leq 6 \},
\label{dim3Mnotmm}
\end{equation}
which is depicted in Fig.~\ref{FGmnatmm}(b).
This set is \Mnat-convex,
but it is neither \Lnat-convex nor multimodular.
For $x=(1,2,3)$ and $y=(2,2,2)$,
we have
$(x+y)/2 = (3/2,2,5/2)$ and
$\lceil (x+y)/2  \rceil = (2,2,3) \notin S$.
Hence $S$ is not \Lnat-convex.
Multimodularity of $S$
is denied in Example~\ref{EXmnatNOTmmdim3} in Section \ref{SCmnatmmth}.
\finbox
\end{example}

\begin{example}  \rm  \label{EXnorelMnotLmmdim4}
From Example~\ref{EXnorelMnatnotLmm},
we can make an example of an M-convex set that is neither \Lnat-convex nor multimodular.
Using $S$ in \eqref{dim3Mnotmm}, define
\begin{equation*} 
\hat S = \{ x \in \ZZ\sp{4} \mid 
(x_{1},x_{2},x_{3}) \in S, \ x_{4}= 6 -  (x_{1}+x_{2}+x_{3}) \},
\end{equation*}
which is indeed M-convex by \eqref{mfnmnatfnrelationvex}.
This set is not \Lnat-convex
since discrete midpoint convexity fails for 
$x=(1,2,3,0)$ and $y=(2,2,2,0)$.
Multimodularity of $\hat S$
is denied in Example~\ref{EXexmnatNOTmmdim4} in Section \ref{SCmnatmmth}.
\finbox
\end{example}

\begin{remark} \rm \label{RMmmMdim23}
When $n=2$,
every \Mnat-convex function is multimodular, and the converse is also true
\cite[Remark 2.2]{MM19multm}. 
When $n=3$, every M-convex function is multimodular
(Proposition~\ref{PRmmMdim3} in Section~\ref{SCmnatmmth}), but the converse is not true.
\finbox
\end{remark}

\subsection{Global and directed discrete midpoint convexity}
\label{SCnorelgddmc}

The following examples show that 
there is no inclusion relation between
global d.m.c.~and directed d.m.c., 
that is,
$A \setminus B \ne \emptyset$
when $\{ A, B \} = \{ \mbox{g-dmc}, \mbox{d-dmc} \}$.

\begin{example}[{\cite[Example~3]{TT21ddmc}}]  \rm  \label{EXnoreldmcNOTddmc}
Let
\[
S=\{(0,0,0),(1,1,0),(1,0,-1),(2,1,-1)\}.
\]
This set is globally d.m.c.~but not directed d.m.c.
Indeed, $x=(0,0,0)$ and $y=(2,1,-1)$
are the only pair with $\| x-y \|_\infty \geq 2$.
We have
$(x+y)/2 = (1,1/2,-1/2)$,
$\lceil (x+y)/2  \rceil = (1,1,0) \in S$ and
$\lfloor (x+y)/2  \rfloor = (1,0,-1) \in S$.
On the other hand, 
using notation $\tilde \mu$ defined in \eqref{ddmcrounddefInt} and \eqref{ddmcrounddefVec},
we have
$\tilde \mu(x,y)=(1,0,0)\notin S$ and
$\tilde \mu(y,x)=(1,1,-1)\notin S$.
Hence $S$ is not directed d.m.c.
\finbox
\end{example}

\begin{example}[{\cite[Example~3]{TT21ddmc}}]  \rm  \label{EXnorelddmcNOTdmc}
Let
\[
S=\{(0,0,0),(1,0,0),(1,1,1),(2,1,1),(1,1,-1), (2, 1,-1), (1, 1, 0), (2, 1, 0) \}.
\]
This set is directed d.m.c.~but not globally d.m.c.
For $x=(0,0,0)$ and $y=(2,1,-1)$
with $\| x-y \|_\infty \geq 2$,
we have
$(x+y)/2 = (1,1/2,-1/2)$ and
$\lfloor (x+y)/2  \rfloor = (1,0,-1) \notin S$.
Hence $S$ is not globally d.m.c.
\finbox
\end{example}

\begin{example}  \rm  \label{EXnorelddmcNOTdmcDiadom1}
Let
\[
 f(x_{1},x_{2},x_{3}) = (x_{1} - x_{2})^2 + (x_{1} + x_{3})^2 + (x_{2} + x_{3})^2.
\]
This function is not globally d.m.c.
Indeed, for $x = (0,0,0)$ and $y = (2,1,-1)$ 
with $\| x-y \|_\infty \geq 2$,
we have
$(x+y)/2 = (1,1/2,-1/2)$,
$u:=\lceil (x+y)/2  \rceil = (1,1,0)$, and
$v:=\lfloor (x+y)/2  \rfloor = (1,0,-1)$.
Then
\[
f(x) + f(y) = (0 + 0 + 0) + (1 + 1 + 0)  < f(u) + f(v) = (0+1+1) + (1+0+1),
\]
which is a violation of discrete midpoint convexity.
This function $f$ is a quadratic function
represented as $f(x) = x\sp{\top} Q x$ with a diagonally dominant matrix
{\small $Q = 
\left[ \begin{array}{rrr}
2 & -1 & 1  \\
-1 & 2 & 1  \\
1 & 1 & 2 \\
\end{array}\right]$}.
It is known 
\cite[Theorem~9]{TT21ddmc}
that a quadratic function
defined by a diagonally dominant matrix with nonnegative diagonal elements
is directed d.m.c.
More generally, 2-separable convex functions are directed d.m.c.
\cite[Theorem~4]{TT21ddmc}.
\finbox
\end{example}

\begin{example}  \rm  \label{EXnorelddmcNOTdmcDiadom2}
Let
\[
 f(x_{1},x_{2},x_{3}) = (x_{1} + x_{2})^2  = 
[ \begin{array}{ccc}
x_{1} & x_{2} & x_{3}
\end{array} ]
\left[ \begin{array}{rrr}
1 & 1 & 0  \\
1 & 1 & 0  \\
0 & 0 & 0 \\
\end{array}\right]
\left[ \begin{array}{c}
x_{1} \\ x_{2} \\ x_{3}
\end{array}\right].
\]
This function is directed d.m.c. (by diagonal dominance)
but not globally d.m.c.
Indeed, for
$x = (1,0,0)$ and $y = (0,1,2)$, we have
$u:= \lceil (x+y)/2  \rceil = (1,1,1)$,
$v := \lfloor (x+y)/2  \rfloor = (0,0,1)$,
and
\[
f(x) + f(y) = 1 + 1  < f(u) + f(v) = 4 + 0,
\]
which is a violation of discrete midpoint convexity.
\finbox
\end{example}

\subsection{\LL-convexity and discrete midpoint convexity}
\label{SCnorelL2dmc}

The following examples show that 
there is no inclusion relation between
\LL-convexity and (global, directed) discrete midpoint convexity, 
that is,
$A \setminus B \ne \emptyset$
when $\{ A, B \} = \{ \mbox{\LL}, \mbox{(g-, d-)dmc} \}$.

\begin{example} [{\cite[Remark~2]{MMTT20dmc}, \cite[Example~2]{TT21ddmc}}] \rm \label{EXnorelLnnotdmc}
Let
\[ 
S=\{(0,0,0),(0,1,1),(1,1,0),(1,2,1)\}.
\]
This set is \LLnat-convex, since $S = S_{1} + S_{2}$ for two \Lnat-convex sets
$S_{1}=\{(0,0,0),(0,1,1)\}$ and $S_{2}=\{(0,0,0),(1,1,0)\}$.
The set $S$ is not (globally, directed) d.m.c.,
since, for $x=(0,0,0)$ and $y=(1,2,1)$
with $x \leq y$ and
$\| x-y \|_\infty \geq 2$,
we have
$(x+y)/2 = (1/2,1, 1/2)$, 
$\lceil (x+y)/2  \rceil = (1,1,1) \notin S$
and
$\lfloor (x+y)/2  \rfloor = (0,1, 0) \notin S$.
From this \LLnat-convex set $S$, 
we define a set $T \subseteq \ZZ\sp{4}$ by
\[ 
 T = \{(x,0) + \mu (\vecone,1) \mid x \in S, \ \mu \in \ZZ \},
\] 
which is \LL-convex and not (globally, directed) d.m.c.
\finbox
\end{example}

\begin{example}  \rm  \label{EXnoreldmcnotL2b}
Let
\begin{align*} 
S &= \{ x \in \ZZ\sp{4}\mid x_{1} + x_{2} - x_{3} - x_{4} =0 \},
\end{align*}
which is \LL-convex, since $S = S_{1} + S_{2}$ for two L-convex sets
\begin{align*} 
S_{1} &= \{ x \in \ZZ\sp{4}\mid x_{1} = x_{3}, \  x_{2} =x_{4}  \},
\nonumber \\ 
S_{2} &= \{ x \in \ZZ\sp{4}\mid x_{1} = x_{4}, \  x_{2} =x_{3}  \}.
\end{align*}
The set $S$ is not (globally, directed) d.m.c.,
since, for $x=(0,0,0,0)$ and $y=(2,2,1,3)$
with $x \leq y$ and
$\| x-y \|_\infty \geq 2$,
we have
$(x+y)/2 = (1,1,1/2,3/2)$, 
$\lceil (x+y)/2  \rceil = (1,1,1,2) \notin S$
and
$\lfloor (x+y)/2  \rfloor = (1,1,0,1) \notin S$.
\finbox
\end{example}

\begin{example}  \rm  \label{EXnorelL2nnotddmc}
The set $S=\{(1,0),(0,1)\}$
is (globally, directed) d.m.c.
but not \LLnat-convex (hence not \LL-convex).
\finbox
\end{example}

\subsection{\MM-convexity and discrete midpoint convexity}
\label{SCnorelM2dmc}

The following examples show that 
there is no inclusion relation between
\MM-convexity and (global, directed) discrete midpoint convexity, 
that is, $A \setminus B \ne \emptyset$
when $\{ A, B \} = \{ \mbox{\MM}, \mbox{(g-, d-)dmc} \}$.

\begin{example}  \rm  \label{EXnorelMnotdmc}
Let $S$ be the set of integer points on the maximal face in Fig.~\ref{FGmnatmm}(a),
that is,
\begin{align}
S & = 
\{ x\in \ZZ^3 \mid  x_i\leq 3~(i=1,2,3), \ 
x_{1}+x_{2}\leq 5, \ x_{2}+x_{3}\leq 5, \ x_{1}+x_{2}+x_{3} = 6 \}
\nonumber \\ &
= \{ (3,0,3), (1,2,3), (3,2,1), (1,3,2), (2,3,1),   (2,1,3), (2,2,2) ,(3,1,2) \}.
\label{mnatmmSmax}
\end{align}
This set $S$ is M-convex and hence \MM-convex.
However, it is not (globally, directed) d.m.c.
For $x=(3,0,3)$ and $y=(2,2,2)$
with $\| x-y \|_\infty \geq 2$,
we have
$(x+y)/2 = (5/2,1,5/2)$,
$\lceil (x+y)/2  \rceil = 
\tilde \mu(x,y) =(3,1,3) \notin S$ and
$\lfloor (x+y)/2  \rfloor =
\tilde \mu(y,x) =(2,1,2) \notin S$
using notation $\tilde \mu$
 defined in \eqref{ddmcrounddefInt} and \eqref{ddmcrounddefVec}.  
Hence $S$ is neither globally d.m.c. nor directed d.m.c.
\finbox
\end{example}

\begin{example}  \rm  \label{EXnoreldmcnotM2}
The set $S=\{ (0,0), (1,1) \}$
is (globally, directed) d.m.c.~but not \MMnat-convex
(hence not \Mnat-convex).
\finbox
\end{example}

\section{Intersection of Two Convexity Classes}
\label{SCinter}

The following are the major findings of this paper
concerning the intersection of two convexity classes
of functions $f:\ZZ\sp{n} \to \Rinf$.
\begin{itemize}
\item 
A function which is both \Lnat -convex and multimodular is separable convex
(Theorem~\ref{THlnatmmB}).

\item 
Furthermore, a function which is both \LLnat -convex and multimodular is separable convex
(Theorem~\ref{THl2natmmB}).

\item 
A function which is both \Mnat-convex and multimodular is not necessarily separable convex
(Section~\ref{SCmnatmm}).

\end{itemize}

For the ease of reference we summarize remarkable relations below:
\begin{align}
 &
 \{ \mbox{\rm \LL-convex} \} \cap  \{ \mbox{\rm \Lnat-convex} \}
= \{ \mbox{\rm L-convex} \} ,
\label{fnfamsepL2Ln}
\\ &
 \{ \mbox{\rm \MM-convex} \} \cap  \{ \mbox{\rm \Mnat-convex} \}
= \{ \mbox{\rm M-convex} \} ,
\label{fnfamsepM2Mn}
\\ &
 \{ \mbox{\rm L-convex} \} \cap  \{ \mbox{\rm M-convex} \} = \emptyset ,
\label{fnfamLMemp}
\\ &
 \{ \mbox{\rm \LL-convex} \} \cap  \{ \mbox{\rm \MM-convex} \} = \emptyset ,
\label{fnfamL2M2emp}
\\ &
 \{ \mbox{\rm \Lnat-convex} \} \cap  \{ \mbox{\rm \Mnat-convex} \}
= \{ \mbox{\rm separable convex} \} ,
\label{fnfamsepLnMn}
\\ &
 \{ \mbox{\rm \LLnat-convex} \} \cap  \{ \mbox{\rm \MMnat-convex} \}
= \{ \mbox{\rm separable convex} \} ,
\label{fnfamsepL2nM2n}
\\ &
\{ \mbox{\rm L-convex} \}  \cap  \{ \mbox{\rm \Mnat-convex} \}
 = \{ \mbox{\rm linear on $\ZZ\sp{n}$} \}, 
\label{fnfamLMn}
\\ &
\{ \mbox{\rm \Lnat-convex} \}  \cap  \{ \mbox{\rm M-convex} \}
 = \{ \mbox{\rm singleton effective domain} \} ,
\label{fnfamLnM}
\\ &
 \{ \mbox{\rm multimodular} \} \cap  \{ \mbox{\rm \Lnat-convex} \}
= \{ \mbox{\rm separable convex} \} ,
\label{fnfamsepmultmLn}
\\ &
 \{ \mbox{\rm multimodular} \} \cap  \{ \mbox{\rm \LLnat-convex} \}
= \{ \mbox{\rm separable convex} \} ,
\label{fnfamsepmultmL2n}
\\ &
 \{ \mbox{\rm multimodular} \} \cap  \{ \mbox{\rm \Mnat-convex} \}
 \supsetneqq \{ \mbox{\rm separable convex} \} ,
\label{fnfamsepmultmMn}
\\ &
 \{ \mbox{\rm globally d.m.c.} \} \cap  \{ \mbox{\rm directed d.m.c.} \}
 \supsetneqq \{ \mbox{\rm \Lnat-convex} \} .
\label{fnfamgdmcddmc}
\end{align}
The relations 
\eqref{fnfamsepL2Ln} and \eqref{fnfamsepM2Mn}
will be established in Theorems \ref{THl2lnatL} and \ref{THm2mnatM}, respectively.
The relations \eqref{fnfamLMemp} and \eqref{fnfamL2M2emp} are obvious.
The relations \eqref{fnfamsepLnMn} and \eqref{fnfamsepL2nM2n} are well known,
originating in \cite[Theorem~3.17]{MS01rel}
and stated in \cite[Theorem~8.49]{Mdcasiam}.
The relations \eqref{fnfamLMn} and \eqref{fnfamLnM} 
will be discussed in Section~\ref{SCrelLM}.
The relations 
\eqref{fnfamsepmultmLn} and \eqref{fnfamsepmultmL2n}
will be established in Theorems \ref{THlnatmmB} and \ref{THl2natmmB}, respectively,
and \eqref{fnfamsepmultmMn} is discussed in Section~\ref{SCmnatmm}.
The relation
\eqref{fnfamgdmcddmc} is treated in Section~\ref{SCreldmc}.

In the following we make observations, major and minor,
for each case of intersection.

\subsection{L-convexity}
\label{SCrelL}

In this section we deal with intersection of classes related to
L-convexity (and its relatives like \Lnat-, \LL-, and \LLnat-convexity).

\begin{itemize}

\item
{\bf L-convexity \& separable convexity}:
An L-convex set $S$ is a box 
if and only if 
$S = \ZZ\sp{n}$.
Hence, a function $f$ is both  L-convex and separable convex  
if and only if 
$f$ is a linear function with $\dom f = \ZZ\sp{n}$.

($\because$
An L-convex set $S$ has translation invariance
($x  \in S, \  \mu \in \ZZ   \Rightarrow  x +  \mu \vecone \in S$)
in \eqref{invarone}.
If a box has this property, it must be $\ZZ\sp{n}$.
The statement for a function follows from the statement for a set.)

\item
{\bf\LL-convexity \& separable convexity}:
An \LL-convex set $S$ is a box
if and only if 
$S = \ZZ\sp{n}$.
Hence, a function $f$ is both  \LL-convex and separable convex  
if and only if 
$f$ is a linear function with $\dom f = \ZZ\sp{n}$.

($\because$
An \LL-convex set $S$ also has translation invariance.)

\item     
{\bf \LL-convexity \& \Lnat-convexity}:
A set $S$ is both \LL-convex and \Lnat-convex
if and only if 
$S$ is L-convex.
Hence, a function $f$ is both  \LL-convex and \Lnat-convex
if and only if 
$f$ is L-convex.
See Theorem~\ref{THl2lnatL} below.
\end{itemize}

The following theorem shows 
that the combination of \LL-convexity and \Lnat-convexity
is equivalent to {\rm L}-convexity.

\begin{theorem} \label{THl2lnatL}
\quad

\noindent
{\rm (1)}
A set $S$ $(\subseteq \ZZ\sp{n})$ is L-convex
if and only if
it is both \LL-convex and \Lnat-convex.

\noindent
{\rm (2)}
A function $f:\ZZ\sp{n} \to \Rinf$
is L-convex
if and only if
it is both \LL-convex and \Lnat-convex.
\end{theorem}

\begin{proof}
(1)
First note that the only-if-part of (1) is obvious.
To prove the if-part, suppose that
a set $S$ is both \LL-convex and \Lnat-convex.
Recall that an \Lnat-convex set $S$ is L-convex 
if and only if
it is translation invariant in the sense that
$x + \mu \vecone \in S$ for every $x \in S$ and every integer $\mu$
(see Theorem~\ref{THlset}).
Since
$S$ is \LL-convex, it can be represented as
$S = S_{1} + S_{2}$ with two L-convex sets $S_{1}$ and $S_{2}$.
Since $S_{1}$ and $S_{2}$ are translation invariant, $S$ is also 
translation invariant. Therefore, $S$ must be L-convex.

(2)
The statement (2) for functions follows from (1) for sets
by the general principle given in 
Proposition~\ref{PRlf-lset}. 
Note that the assumption \eqref{cnvABargminfp} 
holds for \LL-convexity and \Lnat-convexity
(see \cite[Proposition 8.40]{Mdcasiam} for this statement about \LL-convexity).
\end{proof}

\subsection{M-convexity}
\label{SCrelM}

In this section we deal with intersection of classes related to
M-convexity (and its relatives like \Mnat-, \MM-, and \MMnat-convexity).

\begin{itemize}
\item
{\bf M-convexity \& separable convexity}:
An M-convex set $S$ is a box 
if and only if 
$S$ is a singleton (a set consisting of a single point).
Hence, a function $f$ is both  M-convex and separable convex  
if and only if 
$f$ is a function whose effective domain is a singleton.

($\because$
An M-convex set $S$ consists of vectors with a constant component sum,
i.e., $x(N) = y(N)$ for all  $x, y \in S$.
If a box has this property, it must be a singleton.
The statement for a function follows from the statement for a set.)

\item
{\bf \MM-convexity \& separable convexity}:
An \MM-convex set $S$ is a box 
if and only if $S$ is a singleton.
Hence, a function $f$ is both  \MM-convex and separable convex  
if and only if 
$f$ is a function whose effective domain is a singleton.

($\because$
An \MM-convex set $S$ also consists of vectors with a constant component sum.)

\item 
{\bf \MM-convexity \& \Mnat-convexity}:
A set $S$ is 
both \MM-convex and \Mnat-convex
if and only if 
$S$ is M-convex.
Hence, a function $f$ is both  \MM-convex and \Mnat-convex
if and only if 
$f$ is M-convex.
See Theorem~\ref{THm2mnatM} below.
\end{itemize}

The following theorem shows 
that the combination of \MM-convexity and \Mnat-convexity
is equivalent to {\rm M}-convexity.

\begin{theorem} \label{THm2mnatM}
\quad

\noindent
{\rm (1)}
A set $S$ $(\subseteq \ZZ\sp{n})$ is M-convex
if and only if
it is both \MM-convex and \Mnat-convex.

\noindent
{\rm (2)}
A function $f:\ZZ\sp{n} \to \Rinf$
is M-convex
if and only if
it is both \MM-convex and \Mnat-convex.
\end{theorem}

\begin{proof}
(1)
First note that the only-if-part of (1) is obvious.
To prove the if-part, suppose that
a set $S$ is both \MM-convex and \Mnat-convex.
Recall that an \Mnat-convex set $S$ is M-convex 
if and only if
$x(N) = y(N)$ for all  $x, y \in S$
(see Section~\ref{SCmfn}).
On the other hand,
an \MM-convex set has this property.
Therefore, $S$ must be M-convex.

(2)
The statement (2) for functions follows from (1) for sets
by the general principle given in 
Proposition~\ref{PRmf-mset}. 
Note that the assumption \eqref{cnvABargminfp} 
holds for \MM-convexity and \Mnat-convexity
(see \cite[Theorems 8.17]{Mdcasiam} for this statement about \MM-convexity).
\end{proof}

\subsection{L-convexity and M-convexity}
\label{SCrelLM}

In this section we deal with intersection of classes
of L-convexity and M-convexity (including their variants).

\begin{itemize}
\item     
{\bf L-convexity \& M-convexity}:
There exists no set that is both L-convex and M-convex.
Hence, there exists no function that is both L-convex and M-convex.

($\because$
An L-convex set $S$ has translation invariance
($x  \in S, \  \mu \in \ZZ   \Rightarrow  x +  \mu \vecone \in S$)
in \eqref{invarone}.
An M-convex set $S$ consists of vectors with a constant component sum,
i.e., $x(N) = y(N)$ for all  $x, y \in S$.
These two properties are not compatible.)

\item     
{\bf \Lnat-convexity \& \Mnat-convexity}:
A set $S$ is 
both \Lnat-convex and \Mnat-convex
if and only if 
$S$ is a box.
Hence, a function $f$ is both  \Lnat-convex and \Mnat-convex
if and only if 
$f$ is separable convex.
See \cite[Theorem~3.17]{MS01rel} and \cite[Theorem~8.49]{Mdcasiam}.

\item     
{\bf L-convexity \& \Mnat-convexity}:
A set $S$ is 
both L-convex and \Mnat-convex
if and only if 
$S = \ZZ\sp{n}$.
Hence, a function $f$ is both  L-convex and \Mnat-convex  
if and only if 
$f$ is a linear function with $\dom f = \ZZ\sp{n}$.

($\because$
An \Mnat-convex set $S$ has the property that
if $x, y \in S$ and $x \leq y$, then $[x, y]_{\ZZ} \subseteq S$,
whereas an L-convex set has translation invariance.
Hence $S = \ZZ\sp{n}$.)

\item
{\bf \Lnat-convexity \& M-convexity}:
A set $S$ is 
both \Lnat-convex and M-convex
if and only if $S$ is a singleton.
Hence, a function $f$ is both \Lnat-convex and M-convex  
if and only if 
$f$ is a function whose effective domain is a singleton.

($\because$
An \Lnat-convex set $S$ has the property that
$x, y \in S$ implies 
$x \vee y, x \wedge y \in S$,
whereas an M-convex set consists of vectors with a constant component sum.
Hence $S$ must be a singleton.)

\end{itemize}

The four statements above can be extended to \LL-convexity and \MM-convexity
without substantial changes.

\begin{itemize}
\item     
{\bf \LL-convexity \& \MM-convexity}:
There exists no set that is both \LL-convex and \MM-convex.
Hence, there exists no function that is both \LL-convex and \MM-convex.
This implies that there exists no function that is 
both  \LL-convex and M-convex (or L-convex and \MM-convex).

($\because$
An \LL-convex set has translation invariance
and an \MM-convex set consists of vectors with a constant component sum.)

\item     
{\bf \LLnat-convexity \& \MMnat-convexity}:
A set $S$ is 
both \LLnat-convex and \MMnat-convex
if and only if 
$S$ is a box.
Hence, a function $f$ is both  \LLnat-convex and \MMnat-convex
if and only if 
$f$ is separable convex.
See \cite[Theorem~3.17]{MS01rel}, \cite[Theorem~8.49]{Mdcasiam}, and
\cite{MM21l2ineq}. 
This implies that
a function $f$ is both  \LLnat-convex and \Mnat-convex 
(or \MMnat-convex and \Lnat-convex) 
if and only if 
$f$ is separable convex.

\item     
{\bf \LL-convexity \& \MMnat-convexity}:
A set $S$ is 
both \LL-convex and \MMnat-convex
if and only if 
$S = \ZZ\sp{n}$ (proved below).
Hence, a function $f$ is both  \LL-convex and \MMnat-convex  
if and only if 
$f$ is a linear function with $\dom f = \ZZ\sp{n}$.
This implies that
a function $f$ is both  \LL-convex and \Mnat-convex 
(or L-convex and \MMnat-convex) 
if and only if 
$f$ is a linear function with $\dom f = \ZZ\sp{n}$.

($\because$
An \MMnat-convex set $S$ has the property that
if $x, y \in S$ and $x \leq y$, then $[x, y]_{\ZZ} \subseteq S$,
whereas an \LL-convex set has translation invariance.
Hence $S = \ZZ\sp{n}$.)

\item
{\bf \LLnat-convexity \& \MM-convexity}:
A set $S$ is 
both \LLnat-convex and \MM-convex
if and only if 
$S$ is a singleton (proved below).
Hence, a function $f$ is both \LLnat-convex and \MM-convex  
if and only if $\dom f$ is a singleton.
This implies that
a function $f$ is both  \LLnat-convex and M-convex 
(or \Lnat-convex and \MM-convex) 
if and only if $\dom f$ is a singleton.

($\because$
Let $S = S_{1}  + S_{2}$
with \Lnat-convex sets $S_{k}$ for $k=1,2$.
Take $x,y \in S$ and represent them as
$x = x_{1}  + x_{2}$ and
$y = y_{1}  + y_{2}$ 
with $x_{k}, y_{k} \in S_{k}$ for $k=1,2$.
Define 
$z_{k} = x_{k} \vee y_{k}$
for $k=1,2$, and 
$z = z_{1} + z_{2}$.
Since $z_{k} \in S_{k}$ for $k=1,2$, 
we have $z \in S$.
By \MM-convexity of $S$, we have $z(N)=x(N)$, while
\[
z(N) = z_{1}(N) + z_{2}(N)
 = (x_{1} \vee y_{1})(N) + (x_{2} \vee y_{2})(N)
 \geq  x_{1}(N) + x_{2}(N) = x(N).
\] 
It then follows that
$x_{k} \vee y_{k} = x_{k}$
for $k=1,2$,
that is,
$y_{k} \leq x_{k}$
for $k=1,2$.
By symmetry we also have
$y_{k} \geq x_{k}$
for $k=1,2$,
and hence $x = y$.)
\end{itemize}

\subsection{Multimodularity}
\label{SCrelmm}

In this section we deal with intersection of classes related to multimodularity.

\begin{itemize}
\item     
{\bf Multimodularity \& L-convexity}:
A set $S$ is 
both multimodular and L-convex
if and only if 
$S = \ZZ\sp{n}$.
Hence, a function $f$ is 
both multimodular and L-convex
if and only if 
$f$ is a linear function with $\dom f = \ZZ\sp{n}$.

($\because$
A multimodular set $S$ has the property that
if $x, y \in S$ and $x \leq y$, then $[x, y]_{\ZZ} \subseteq S$
(Proposition~\ref{PRmmbox}),
whereas an L-convex set has translation invariance.
Hence $S = \ZZ\sp{n}$.)

\item     
{\bf Multimodularity \& \LL-convexity}:
A set $S$ is 
both multimodular and \LL-convex
if and only if 
$S = \ZZ\sp{n}$.
Hence, a function $f$ is 
both multimodular and \LL-convex
if and only if 
$f$ is a linear function with $\dom f = \ZZ\sp{n}$.

($\because$
A multimodular set $S$ has the property that
if $x, y \in S$ and $x \leq y$, then $[x, y]_{\ZZ} \subseteq S$
(Proposition~\ref{PRmmbox}),
whereas an \LL-convex set has translation invariance.
Hence $S = \ZZ\sp{n}$.)

\item     
{\bf Multimodularity \& \Lnat-convexity}:
A set $S$ is 
both multimodular and \Lnat-convex
if and only if 
$S$ is a box.
Hence, a function $f$ is both multimodular and \Lnat-convex
if and only if 
$f$ is separable convex.
See Theorem~\ref{THlnatmmB} in Section~\ref{SCl2natmm}.

\item     
{\bf Multimodularity \& \LLnat-convexity}:
A set $S$ is 
both multimodular and \LLnat-convex
if and only if 
$S$ is a box.
Hence, a function $f$ is both multimodular and \LLnat-convex
if and only if 
$f$ is separable convex.
See Theorem~\ref{THl2natmmB} in Section~\ref{SCl2natmm}.

\item 
{\bf Multimodularity \& \Mnat-convexity}:
A function in two variables is multimodular
if and only if it is \Mnat-convex
\cite[Remark 2.2]{MM19multm}. 
For functions in more than two variables,
these classes are distinct, without inclusion relation.
The intersection of these classes
can be analyzed and understood fairly well (Section~\ref{SCmnatmmth}). 
Figure~\ref{FGmnatmm}(a) demonstrates their nonempty intersection, 
giving a concrete example of a set that is both multimodular and \Mnat-convex
(see Example~\ref{EXmnatmmdim3}). 
This is also an instance of a set that is both multimodular and \MMnat-convex.

\item 
{\bf Multimodularity \& M-convexity}:
The intersection of these two classes
can be analyzed and understood fairly well
(Section~\ref{SCmnatmmth}). 
Their intersection is  nonempty,
which is demonstrated by 
the set $S$ in \eqref{mnatmmSmax}.
That is,
the set of integer points on the maximal face 
($x_{1}+x_{2}+x_{3}=6$) in Fig.~\ref{FGmnatmm}(a)
is both multimodular and M-convex
(see Example~\ref{EXmnatmmdim3}).
This is also an instance of a set that
is both multimodular and \MM-convex.
\end{itemize}

\subsection{Discrete midpoint convexity}
\label{SCreldmc}

In this section we deal with intersection of classes related to
discrete midpoint convexity.
We first note that the indicator function of any nonempty subset of $\{ 0, 1 \}\sp{n}$
is (globally, directed) discrete midpoint convex.
We also note that the indicator function of $S= \{ (2,0), (1,1), (0,2) \}$
is (globally, directed) discrete midpoint convex.

\begin{itemize}

\item 
{\bf d.m.c. \& \LL-convexity}:
An L-convex function is both \LL-convex and  (globally, directed) discrete midpoint convex,
but it is not known whether the converse is true or not.

\item 
{\bf d.m.c. \& \LLnat-convexity}:
An \Lnat-convex function is both \LLnat-convex and (globally, directed) discrete midpoint convex,
but it is not known whether the converse is true or not.

\item 
{\bf d.m.c. \& M-convexity}:
For every matroid, 
the indicator function of the set of its bases 
is both M-convex and (globally, directed) discrete midpoint convex.
The indicator function of $S= \{ (2,0), (1,1), (0,2) \}$
is both M-convex and (globally, directed) discrete midpoint convex.
Note that general separable convex functions do not belong to this class 
because a separable convex function $f$ is M-convex only if $\dom f$ is a singleton.

\item 
{\bf d.m.c. \& \MM-convexity}:
For every pair of matroids, 
the indicator function of the set of their common bases 
is both \MM-convex and (globally, directed) discrete midpoint convex.
For example,
the indicator function of $S= \{ (1,1,0,0), (0,0,1,1) \}$
is both \MM-convex
and (globally, directed) discrete midpoint convex.

\item 
{\bf d.m.c. \& \Mnat-convexity}:
A separable convex function is both \Mnat-convex and 
(globally, directed) discrete midpoint convex.
The converse is not true.
Indeed, the set $S= \{ (1,0), (0,1) \}$
is \Mnat-convex and (globally, directed) discrete midpoint convex,
but it is not a box.
Another example is given by $S= \{ (2,0), (1,1), (0,2) \}$.

\item 
{\bf d.m.c. \& \MMnat-convexity}:
A separable convex function is both \MMnat-convex and 
(globally, directed) discrete midpoint convex.
The converse is not true,  as shown by the indicator function of 
$S= \{ (1,0), (0,1) \}$ or $S= \{ (2,0), (1,1), (0,2) \}$.

\item 
{\bf d.m.c. \& multimodularity}:
A separable convex function is both multimodular and 
(globally, directed) discrete midpoint convex.
The converse is not true,  as shown by the indicator function of
$S= \{ (1,0), (0,1) \}$ or $S= \{ (2,0), (1,1), (0,2) \}$.

\item 
{\bf globally d.m.c. \& directed d.m.c.}:
An \Lnat-convex function is both globally and directed discrete midpoint convex.
The converse is not true,  as shown by the indicator function of
$S= \{ (1,0), (0,1) \}$ or $S= \{ (2,0), (1,1), (0,2) \}$.
\end{itemize}

\section{Results Concerning Multimodularity}
\label{SCresult}

In this section we present 
results concerning multimodularity,
which constitute the major technical contributions
of this paper.

\subsection{Definition of multimodularity}
\label{SCmmfnDef}

Recall that $\unitvec{i}$ denotes the $i$th unit vector for $i =1,2,\ldots, n$, and 
let $\mathcal{F} \subseteq \ZZ\sp{n}$ be a
set of vectors defined by
\begin{equation} \label{multimodirection1}
\mathcal{F} = \{ -\unitvec{1}, \unitvec{1}-\unitvec{2}, \unitvec{2}-\unitvec{3}, \ldots, 
  \unitvec{n-1}-\unitvec{n}, \unitvec{n} \} .
\end{equation}
A finite-valued function $f: \ZZ\sp{n} \to \RR$
is said to be {\em multimodular}
\cite{Haj85}
if it satisfies
\begin{equation} \label{multimodulardef1}
 f(z+d) + f(z+d') \geq   f(z) + f(z+d+d')
\end{equation}
for all 
$z \in \ZZ\sp{n}$ and all distinct $d, d' \in \mathcal{F}$.
It is known \cite[Proposition~2.2]{Haj85} 
that $f: \ZZ\sp{n} \to \RR$
is multimodular if and only if the function 
$\tilde f: \ZZ\sp{n+1} \to \RR$ 
defined by 
\begin{equation} \label{multimodular1}
 \tilde f(x_{0}, x) = f(x_{1}-x_{0},  x_{2}-x_{1}, \ldots, x_{n}-x_{n-1})  
 \qquad ( x_{0} \in \ZZ, x \in \ZZ\sp{n})
\end{equation}
is submodular in $n+1$ variables. 
This characterization enables us to define multimodularity 
for a function that may take the infinite value $+\infty$.
That is, we say
\cite{MM19multm,Mmult05} 
that a function $f: \ZZ\sp{n} \to \RR \cup \{ +\infty \}$ with $\dom f \not= \emptyset$
is multimodular if the function 
$\tilde f: \ZZ\sp{n+1} \to \RR \cup \{ +\infty \}$
associated with $f$ by \eqref{multimodular1} is submodular.

Multimodularity and L$\sp{\natural}$-convexity
have the following close relationship.

\begin{theorem}[\cite{Mmult05}]  \label{THmmfnlnatfn}
A function $f: \ZZ\sp{n} \to \RR \cup \{ +\infty \}$
is multimodular if and only if the function $g: \ZZ\sp{n} \to \RR \cup \{ +\infty \}$
defined by
\begin{equation} \label{mmfnGbyF}
 g(y) = f(y_{1}, \  y_{2}-y_{1}, \  y_{3}-y_{2}, \ldots, y_{n}-y_{n-1})  
 \qquad ( y \in \ZZ\sp{n})
\end{equation}
is L$\sp{\natural}$-convex.
\finbox
\end{theorem}

The relation \eqref{mmfnGbyF} between $f$ and $g$ can be rewritten as
\begin{equation} \label{mmfnFbyG}
 f(x) = g(x_{1}, \  x_{1}+x_{2}, \  x_{1}+x_{2}+x_{3}, 
   \ldots, x_{1}+ \cdots + x_{n})  
 \qquad ( x \in \ZZ\sp{n}) .
\end{equation}
Using a bidiagonal matrix 
$D=(d_{ij} \mid 1 \leq i,j \leq n)$ defined by
\begin{equation} \label{matDdef}
 d_{ii}=1 \quad (i=1,2,\ldots,n),
\qquad
 d_{i+1,i}=-1 \quad (i=1,2,\ldots,n-1),
\end{equation}
we can express \eqref{mmfnGbyF} and \eqref{mmfnFbyG} 
more compactly  as $g(y)=f(Dy)$ and $f(x)=g(D\sp{-1}x)$, respectively. 
The matrix $D$ is unimodular, and its inverse
$D\sp{-1}$ is an integer lower-triangular matrix
with $(D\sp{-1})_{ij}=1$ for $i \geq j$ and 
$(D\sp{-1})_{ij}=0$ for $i < j$.
For $n=4$, for example, we have
\[
D = {\small
\left[ \begin{array}{rrrr}
1 & 0 & 0 & 0 \\
-1 & 1 & 0 & 0 \\
0 & -1 & 1 & 0 \\
0 & 0 & -1 & 1 \\
\end{array}\right]},
\qquad
D\sp{-1} = {\small
\left[ \begin{array}{rrrr}
1 & 0 & 0 & 0 \\
1 & 1 & 0 & 0 \\
1 & 1 & 1 & 0 \\
1 & 1 & 1 & 1 \\
\end{array}\right]}.
\]

A nonempty set $S$ is called {\em multimodular}
if its indicator function $\delta_{S}$ is multimodular. 
The effective domain of a multimodular function is a multimodular set.
The following proposition, a special case of Theorem~\ref{THmmfnlnatfn},
connects multimodular sets with \Lnat-convex sets.

\begin{proposition}[\cite{MM19multm,Mmult05}]  \label{PRmultmsetLnatset}
A set $S  \subseteq \ZZ\sp{n}$ is multimodular 
if and only if 
it can be represented as  $S = \{ D y \mid y \in T \}$
for some \Lnat-convex set $T$,
where $T$ is uniquely determined from $S$ as 
$T = \{ D\sp{-1} x \mid x \in S \}$.
\finbox
\end{proposition}

\begin{theorem}\label{THmmf-fp}
Under the assumption \eqref{ABCfunAssm},
a function $f:\ZZ\sp{n} \to \Rinf$ 
is multimodular 
if and only if,
for any vector $p \in \RR\sp{n}$, $\argmin f[-p]$ is a multimodular set
or an empty set.
\end{theorem}
\begin{proof}
By Theorem~\ref{THmmfnlnatfn} and Proposition~\ref{PRmultmsetLnatset},
this follows from the corresponding statement (Theorem~\ref{THlf-fp})  
for \Lnat-convex functions.
\end{proof}

\begin{example}  \rm  \label{EXnorelLnotMmm-2}
Proposition~\ref{PRmultmsetLnatset} is applied to the set
$S = \{ x \in \ZZ\sp{3}  \mid x= t (1,1,1), t \in \ZZ \}$
treated in Example~\ref{EXnorelLnotMmm}.  
We have 
\[
T = \{ D\sp{-1} x \mid x \in S \} =  \{ x \in \ZZ\sp{3} \mid x= t (1, 2 ,3), t \in \ZZ \},
\]
which is not \Lnat-convex. 
Hence $S$ is not multimodular.
\finbox
\end{example}

The reader is referred to 
\cite{AGH00,AGH03,Haj85,MM19multm,Mmult05,Mdcaprimer07}
for more about multimodularity.


\subsection{Polyhedral description of multimodular sets}
\label{SCmmpolydesc}

As a technical basis for our argument in 
Sections \ref{SCl2natmm} and \ref{SCmnatmm},
we show a description of multimodular sets
by inequalities.
A subset $I$ of the index set $N = \{ 1,2,\ldots, n \}$
is said to be {\em consecutive}
if it consists of consecutive numbers, that is,
it is a set of the form 
$I = \{ k, k+1, \ldots, l -1, l \}$
for some $k \leq l$.

\begin{theorem} \label{THmmpolydesc}
\quad

\noindent
{\rm (1)}
If $S \subseteq \mathbb{Z}\sp{n}$ is a multimodular set, then 
$S = \overline{S} \cap \ZZ\sp{n}$ and
its convex hull $\overline{S}$ is represented as
\begin{equation} \label{mmineqSbar}
  \overline{S} =   \{ x \in \RR\sp{n} \mid  a_{I} \leq x(I) \leq b_{I} 
 \ \ (\mbox{\rm $I$: consecutive subset of $N$})   \} ,
\end{equation}
where $a_{I}=\inf \{ x(I) \mid x \in S  \}$ and
$b_{I}=\sup \{ x(I) \mid x \in S  \}$
for consecutive subsets $I$ of $N$.

\noindent
{\rm (2)}
For any $a_{I} \in \ZZ \cup \{ -\infty \}$ and $b_{I} \in \ZZ \cup \{ +\infty \}$
indexed by the family of consecutive subsets $I$ of $N$,
the polyhedron $P$ defined by
\begin{equation} \label{mmineqR}
 P =   \{ x \in \RR\sp{n} \mid  a_{I} \leq x(I) \leq b_{I} 
 \ \ (\mbox{\rm $I$: consecutive subset of $N$})   \}
\end{equation}
is an integer polyhedron, and 
$S:= P \cap \ZZ\sp{n}$ is a multimodular set, provided $P \ne \emptyset$.

\noindent
{\rm (3)}
A nonempty set $S \subseteq \mathbb{Z}\sp{n}$ is multimodular
if and only if it can be represented as
\begin{equation} \label{mmineqZ}
S =   \{ x \in \ZZ\sp{n} 
\mid  a_{I} \leq x(I) \leq b_{I}  \ \ (\mbox{\rm $I$: consecutive subset of $N$})  \}
\end{equation}
for some $a_{I} \in \ZZ \cup \{ -\infty \}$ and $b_{I} \in \ZZ \cup \{ +\infty \}$
indexed by the family of consecutive subsets $I$ of $N$.
\end{theorem}
\begin{proof}
First we mention some basic facts about the transformation $x=D y$
connecting multimodularity and \Lnat-convexity.
Define  
$T = \{ y \mid y=D\sp{-1} x,  \ x \in S \}$
for any nonempty set $S \subseteq \ZZ\sp{n}$.
Then the convex hulls of $T$ and $S$ correspond to each other
by the same transformation, that is,
$\overline{T} = \{ y \mid y=D\sp{-1} x,  \ x \in \overline{S} \}$.
Since $D$ is unimodular,
$T = \overline{T} \cap \ZZ\sp{n}$ 
holds if and only if
$S = \overline{S} \cap \ZZ\sp{n}$.
By Theorem~\ref{THlnatpolydesc},
$T$ is \Lnat-convex
if and only if
$T = \overline{T} \cap \ZZ\sp{n}$ and
its convex hull $\overline{T}$ can be described as
\begin{align*} 
 & 
\alpha_{i} \leq y_{i} \leq \beta_{i} \quad (i=1,2,\ldots,n), 
\\ &
\alpha_{ij} \leq y_{j} - y_{i} \leq \beta_{ij} \quad (1 \leq i < j \leq n). 
\end{align*} 
By substituting
$y_{i} = x_{1} + x_{2} + \cdots + x_{i}$
$(i= 1,2,\ldots,n)$, i.e., $y=D\sp{-1} x$, 
into these inequalities we obtain 
the following inequalities 
\begin{align} 
 & 
\alpha_{i} \leq x_{1} + x_{2} + \cdots + x_{i} \leq \beta_{i} \quad (i=1,2,\ldots,n), 
\label{aix1ibi}
\\ &
\alpha_{ij} \leq x_{i+1} + x_{i+2} + \cdots + x_{j} \leq \beta_{ij} 
\quad 
(1 \leq i < j \leq n) 
\label{aijxijbij}
\end{align}
to describe $\overline{S}$.

(1)
If $S$ is multimodular, then
$T$ is \Lnat-convex
by Proposition~\ref{PRmultmsetLnatset}.
By the above argument, $\overline{S}$ is described 
by \eqref{aix1ibi} and \eqref{aijxijbij}.
These inequalities are of the form of 
$a_{I} \leq x(I) \leq b_{I}$
for consecutive subsets $I$;
\eqref{aix1ibi} corresponding to 
$I= \{ 1,2,\ldots,i \}$
and \eqref{aijxijbij} to 
$I= \{ i+1, i+2,\ldots,j \}$.
Thus we obtain \eqref{mmineqSbar}.
Then the validity of the expressions
$a_{I}=\inf \{ x(I) \mid x \in S  \}$ and
$b_{I}=\sup \{ x(I) \mid x \in S  \}$
is obvious.

(2)
Let  $\mathcal{I}_{1}$ (resp., $\mathcal{I}_{2}$)
denote the family of consecutive subsets $I$
for which $b_{I}$  (resp., $a_{I}$) is finite.
Let  
$A_{1}$  (resp., $A_{2}$) 
denote the matrix whose rows are indexed by
$\mathcal{I}_{1}$ (resp., $\mathcal{I}_{2}$)
and whose row vector corresponding to $I$
is the characteristic vector $\unitvec{I}$.
Let  $b$ (resp., $a$) be the vector 
$(b_{I} \mid I \in \mathcal{I}_{1})$ 
(resp., $(a_{I} \mid I \in \mathcal{I}_{2})$). 
Then the system of inequalities in \eqref{mmineqR} 
can be expressed as
\begin{equation}  \label{A1A2ba}
\left[
\begin{array}{r}
  A_{1} \\
  -A_{2} \\
\end{array} 
\right]
x \leq 
\left[
\begin{array}{r}
    b \\
   -a   \\
\end{array} 
\right] .
\end{equation}
Each of the matrices $A_{1}$ and $A_{2}$ is a so-called
interval matrix  
(row-oriented),
or a matrix with consecutive ones property.
This implies that the combined matrix
{\small
$\left[
\begin{array}{r}
  A_{1} \\
 A_{2} \\
\end{array} 
\right]$
} 
is also an interval matrix,
which is known to be totally unimodular \cite[p.279, Example~7]{Sch86}.
Therefore, the coefficient matrix
{\small
$\left[
\begin{array}{r}
  A_{1} \\
 -A_{2} \\
\end{array} 
\right]$
} 
is also totally unimodular,
and hence \eqref{A1A2ba} determines an integer polyhedron,
which shows the integrality of $P$
in \eqref{mmineqR}.
The integrality of $P$ implies 
$P = \overline{P \cap \ZZ\sp{n}}$, 
which is equivalent to saying that 
the convex hull of 
$S= P \cap \ZZ\sp{n}$ 
is described 
by \eqref{aix1ibi} and \eqref{aijxijbij}.
Then the basic facts presented at the beginning of the proof shows that
$T = \{ y \mid y=D\sp{-1} x,  \ x \in S \}$
is \Lnat-convex, 
which implies, 
by Proposition~\ref{PRmultmsetLnatset},
that $S$ is multimodular.

(3)
If $S$ is multimodular, the statement (1) shows that
$S = \overline{S} \cap \ZZ\sp{n}$
and $\overline{S}$ is described as \eqref{mmineqSbar}.
Hence follows \eqref{mmineqZ}.
Conversely, suppose that \eqref{mmineqZ} holds.
Then
 $P:= \overline{S}$ is described as \eqref{mmineqR}
by the integrality of a polyhedron of the form of \eqref{mmineqR}
shown in the statement (2).
The expression \eqref{mmineqZ} implies
$S = \overline{S} \cap \ZZ\sp{n} =P \cap \ZZ\sp{n}$,
while $P \cap \ZZ\sp{n}$ is multimodular by the statement in (2).
Thus $S$ is multimodular if it is represented as \eqref{mmineqZ}.
\end{proof}

It is worth while mentioning box-total dual integrality
of the inequality system given in Theorem~\ref{THmmpolydesc}. 
A linear inequality system $Ax \leq  b$
(in general)
is said
 to be {\em totally dual integral} 
({\em TDI}) if the entries of $A$ and $b$ are rational numbers and
the minimum in the linear programming duality equation
\begin{equation*}  
  \max\{w\sp{\top} x \mid  Ax \leq b \} \ 
  = \ \min \{y\sp{\top} b \mid  y\sp{\top} A=w\sp{\top}, \ y\geq 0 \}
\end{equation*}
has an integral optimal solution $y$ for every integral vector $w$ 
such that the minimum is finite
(\cite{Sch86,Sch03}).
A linear inequality system $Ax \leq  b$
is said to be {\em box-totally dual integral}  ({\em box-TDI}) 
if the system $[ Ax \leq  b, d \leq x\leq c ]$ 
is TDI for each choice of rational (finite-valued) vectors $c$ and $d$.  
It is known  \cite[Theorem 5.35]{Sch03} that
the system $A x \leq  b$ is box-TDI
if the matrix $A$ is totally unimodular.
A polyhedron is called a {\em box-TDI polyhedron} if it can be
described by a box-TDI system.

\begin{theorem} \label{THmmboxTDI}
The inequality system \eqref{mmineqR} describing 
the convex hull of a multimodular set is box-TDI,
and the convex hull of a multimodular set is a box-TDI integer polyhedron.
\end{theorem}

\begin{proof}
In the proof of Theorem~\ref{THmmpolydesc},  the matrix
{\small
$\left[
\begin{array}{r}
  A_{1} \\
 -A_{2} \\
\end{array} 
\right]$
} 
in \eqref{A1A2ba} is totally unimodular.  
This implies that the inequality system \eqref{A1A2ba} is box-TDI.
\end{proof}

Theorem~\ref{THmmpolydesc}
implies immediately that a box of integers is a multimodular set,
which was pointed out first in \cite[Proposition~2]{MM19multm}.
Another immediate consequence of Theorem~\ref{THmmpolydesc}
is the following property of multimodular sets.
This fact will be used in 
the proof of Theorem~\ref{THl2natmmB} in Section~\ref{SCl2natmm}.

\begin{proposition} \label{PRmmbox}
Let $S$ be a multimodular set.
If $x, y \in S$ and $x \leq y$, then $[x, y]_{\ZZ} \subseteq S$.
\end{proposition}
\begin{proof}
Recall the description of $S$ in \eqref{mmineqZ}.
For any $z \in [x, y]_{\ZZ}$ we have
$a_{I} \leq x(I) \leq z(I) \leq y(I) \leq b_{I}$  
for every consecutive index set $I$, which shows $z \in S$.
\end{proof}

\begin{example}  \rm  \label{EXnorelmmnotLM-2}
We show the multimodularity of
$S= \{ x \in \ZZ\sp{3}  \mid x= t (1,-1,1), t \in \ZZ \}$,
which was stated in Example~\ref{EXnorelmmnotLM} without proof.
This set admits a representation 
\[
 S = \{ x \in \ZZ\sp{3} \mid  x_{1}+x_{2}=0, \ x_{2}+x_{3}=0 \}
\]
of the form of \eqref{mmineqZ}, 
and therefore it is multimodular by Theorem~\ref{THmmpolydesc}.
Multimodularity of $S$ can also be verified by 
Proposition~\ref{PRmultmsetLnatset}.
Indeed we have 
\[
T = \{ D\sp{-1} x \mid x \in S \} =  \{ x \in \ZZ\sp{3} \mid x= t (1, 0 ,1), t \in \ZZ \},
\]
which is \Lnat-convex. 
\finbox
\end{example}

\subsection{Multimodularity and \LLnat-convexity}
\label{SCl2natmm}

The following theorem shows 
the combination of multimodularity and \LLnat-convexity
is equivalent to separable convexity.

\begin{theorem} \label{THl2natmmB}
\quad

\noindent
{\rm (1)}
A set $S$ $(\subseteq \ZZ\sp{n})$ is a box of integers
if and only if
it is both multimodular and \LLnat-convex.

\noindent
{\rm (2)}
A function $f:\ZZ\sp{n} \to \Rinf$
is separable convex
if and only if
it is both multimodular and \LLnat-convex.
\end{theorem}

\begin{proof}
(1)
First note that the only-if-part of (1) is obvious.
The if-part can be shown as follows.
Let $S_{1}, S_{2} \subseteq \ZZ\sp{n}$ be \Lnat-convex sets such that
$S = S_{1} + S_{2}$.

In the case of bounded $S$, the proof is quite easy as follows.
Each $S_k$ has the unique minimum element $a^k \in S_k$ 
and the unique maximum element $b^k \in S_k$.
Then $a = a^1 + a^2$ is the unique minimum of $S$
and
$b = b^1 + b^2$ is the unique maximum of $S$,
for which we have
$S \subseteq [a, b]_{\ZZ}$.
By Proposition~\ref{PRmmbox},
on the other hand, it follows from $a, b \in S$ and $a \leq b$ that 
$[a, b]_{\ZZ} \subseteq S$.
Thus we have proved $S = [a, b]_{\ZZ}$.

The general case with (possibly) unbounded $S$ can be treated as follows.
For each $i \in N$, put $a_{i} = \inf_{y \in S}y_{i}$ and
$b_{i} = \sup_{y \in S}y_{i}$,
where we have the possibility of 
$a_{i}=-\infty$ and/or 
$b_{i}=+\infty$.
 Obviously, $S \subseteq [a, b]_{\ZZ}$ holds.
 To prove $[a, b]_{\ZZ} \subseteq S$, take any $x \in [a, b]_{\ZZ}$. 
For each $i \in N$, there exist vectors $p\sp{i}, q\sp{i} \in S$
such that $p\sp{i}_{i} \leq x_{i} \leq q\sp{i}_{i}$,
where $p\sp{i}_{i}$, $x_{i}$, and $q\sp{i}_{i}$ denote the 
$i$th component
 of vectors $p\sp{i}$, $x$, and $q\sp{i}$, respectively.
Since $p\sp{i}, q\sp{i} \in S = S_{1} + S_{2}$,
we can express them as
$p\sp{i} = p\sp{i1} + p\sp{i2}$, 
$q\sp{i} = q\sp{i1} + q\sp{i2}$
with some $p\sp{ik}, q\sp{ik} \in S_k$ $(k = 1, 2)$.
Consider
\[
\begin{array}{ll}
\displaystyle 
\hat{p}\sp{k} = \bigwedge_{i \in N}p\sp{ik} \in S_k \quad (k = 1, 2), 
\qquad  \hat{p} = \hat{p}\sp{1} + \hat{p}\sp{2} \in S,\\
\displaystyle 
\hat{q}\sp{k} = \bigvee_{i \in N}q\sp{ik} \in S_k \quad (k = 1, 2), 
\qquad  \hat{q} = \hat{q}\sp{1} + \hat{q}\sp{2} \in S.
\end{array}
\]
Then, for each component $i \in N$,we have 
\begin{align*}
 \hat{p}_{i} &= \hat{p}\sp{1}_{i} + \hat{p}\sp{2}_{i} \leq p\sp{i1}_{i} + p\sp{i2}_{i} 
  = p\sp{i}_{i} \leq x_{i} ,
\\
 \hat{q}_{i} &= \hat{q}\sp{1}_{i} + \hat{q}\sp{2}_{i} \geq q\sp{i1}_{i} + q\sp{i2}_{i} 
 = q\sp{i}_{i} \geq x_{i},
\end{align*}
which shows $x \in [\hat{p}, \hat{q}]_{\ZZ}$.
By Proposition~\ref{PRmmbox},
it follows from $\hat{p}, \hat{q} \in S$ and $\hat{p} \leq \hat{q}$ that 
$[\hat{p}, \hat{q}]_{\ZZ} \subseteq S$.
Therefore, $x \in [\hat{p}, \hat{q}]_{\ZZ} \subseteq S$,
where $x$ is an arbitrarily chosen element of $[a, b]_{\ZZ}$.
Hence
$[a, b]_{\ZZ} \subseteq S$.
Thus we complete the proof of $S = [a, b]_{\ZZ}$.

(2)
The statement (2) for functions follows from (1) for sets
by the general principle given in 
Proposition~\ref{PRscf-box}. 
Note that the assumption \eqref{cnvABargminfp} 
holds for multimodularity and \LLnat-convexity.
\end{proof}

The following statement, with \Lnat-convexity
in place of \LLnat-convexity,
is an immediate corollary of Theorem~\ref{THl2natmmB},
since separable convex functions are \Lnat-convex.

\begin{theorem} \label{THlnatmmB}
\quad

\noindent
{\rm (1)}
A set $S$ $(\subseteq \ZZ\sp{n})$ is a box of integers
if and only if
it is both multimodular and \Lnat-convex.

\noindent
{\rm (2)}
A function $f:\ZZ\sp{n} \to \Rinf$
is separable convex
if and only if
it is both multimodular and \Lnat-convex.
\finbox
\end{theorem}

\subsection{Multimodularity and \Mnat-convexity}
\label{SCmnatmm}

\subsubsection{Theorem}
\label{SCmnatmmth}

For $a,b \in N = \{ 1,2,\ldots, n  \}$,
we denote the set of integers between $a$ and $b$ by 
$I(a,b)$, that is,
\[
I(a,b) = \{ i \in \ZZ \mid a \leq i \leq b \}.
\]
In particular, $I(a,b) = \emptyset$ if $a >b$.
We consider an integer-valued function
$r(a,b)$, defined for all
$(a,b)$ with $1 \leq a \leq b \leq n$,
satisfying the following conditions:
\begin{align}
&r(a,a) \geq 0
\qquad (a \in N),
\label{mmMr0} \\ &
\max\{ r(a,b-1) , r(a+1,b) \} \leq r(a,b)
\qquad (a < b),
\label{mmMr1} \\ &
 r(a,b) \leq r(a,b-1) + r(a+1,b) - r(a+1,b-1)
\qquad (a < b).
\label{mmMr2} 
\end{align}

Recall the relation of \Mnat-convexity and polymatroids that
the set of integer points of an integral polymatroid
is precisely a bounded \Mnat-convex set 
containing the origin $\veczero$ and 
contained in the nonnegative orthant.  
The following theorem characterizes an integral polymatroid
that is also multimodular.
The proof is given in Section~\ref{SCmnatmmprf}.

\begin{theorem} \label{THmmMnat}
A bounded set $S$ 
containing the origin $\veczero$ and 
contained in the nonnegative orthant 
$(\veczero \in S \subseteq \ZZ_{+}\sp{n})$
is both \Mnat-convex and multimodular
if and only if it is described as 
\begin{equation} \label{mmMnatpoly}
 S= \{ x \in \ZZ_{+}\sp{n} \mid 
  x(I(a,b)) \leq r(a,b) \ (1 \leq a \leq b \leq n) \}
\end{equation}
with a function $r$ satisfying \eqref{mmMr0}--\eqref{mmMr2}.
\finbox
\end{theorem}

\begin{remark} \rm \label{RMfnrab}
The function $r$ can be interpreted as a set function $\rho$
defined for consecutive subsets of $N$
by $\rho(I(a,b))=r(a,b)$.
The conditions
\eqref{mmMr0}, \eqref{mmMr1}, and \eqref{mmMr2} 
correspond, respectively, to nonnegativity, monotonicity, and submodularity of $\rho$
(see Section~\ref{SCmnatmmprf} for the precise meaning).
Every function $r$ satisfying \eqref{mmMr0}--\eqref{mmMr2} 
can be constructed as follows.
First assign arbitrary nonnegative integers to $r(a,a)$ for $a \in N$.
Next, for each $a \in N - \{ n \}$, 
assign to $r(a,a+1)$ an arbitrary integer between
$\max\{ r(a,a), r(a+1,a+1) \}$ and $r(a,a) + r(a+1,a+1)$.
Then, for $k=2,3,\ldots,n-1$, define $r(a,b)$ with $b-a=k$
that satisfy \eqref{mmMr1} and \eqref{mmMr2},
which is possible because the right-hand side of \eqref{mmMr2} 
is not smaller than the left-hand side of \eqref{mmMr1}. 
\finbox
\end{remark}

It follows from Theorem~\ref{THmmMnat} with 
Theorem~\ref{THmmpolydesc} that, 
if a bounded set $S$ is \Mnat-convex and multimodular, 
then the set of its maximal elements 
$S_{\max}$ is  M-convex and multimodular
(Example~\ref{EXmnatmmdim3}).
However, the converse is not true
(Example~\ref{EXmnatNOTmmdim3}).

\begin{example}   \rm  \label{EXmnatmmdim3}
Let $S$ be the set of integer points of the polymatroid in Fig.~\ref{FGmnatmm}(a).
This set can be described in the form of \eqref{mmMnatpoly} as
\[
S= 
\{x\in \ZZ_{+}^3 \mid  x_i\leq 3~(i=1,2,3), \ 
x_{1}+x_{2}\leq 5, \ x_{2}+x_{3}\leq 5, \ x_{1}+x_{2}+x_{3} \leq 6 \},
\]
which does not involve an inequality for $x_{1}+x_{3}$ corresponding to a non-consecutive 
subset $\{ 1, 3 \}$.
By Theorem~\ref{THmmMnat}, $S$ is both \Mnat-convex and multimodular.
The set of the maximal elements
\[
S_{\max}= \{ (3,0,3), (1,2,3), (3,2,1), (1,3,2), (2,3,1),      (2,1,3), (2,2,2) ,(3,1,2) \}
\]
is both M-convex and multimodular.
\finbox
\end{example}

\begin{example}   \rm  \label{EXmnatNOTmmdim3}
Let $S$ be the set of integer points of the polymatroid in Fig.~\ref{FGmnatmm}(b).
This set can be described as
\begin{align}
S= &
\{x\in \ZZ_{+}^3 \mid  x_i\leq 3~(i=1,2,3), \ 
x_{1}+x_{2}\leq 5, \ x_{1}+x_{3}\leq 5,
\nonumber \\ & \phantom{AAAAAA}
 x_{2}+x_{3}\leq 5, \ x_{1}+x_{2}+x_{3} \leq 6 \}.
\label{mnatNOTmmdim3}
\end{align}  
Since this expression involves a (non-redundant) inequality for $x_{1}+x_{3}$ 
corresponding to a non-consecutive subset $\{ 1, 3 \}$,
Theorem~\ref{THmmMnat}
(or Theorem~\ref{THmmpolydesc})
 shows that $S$ is not multimodular.
While $S$ is not multimodular,
the set of its maximal elements 
\begin{align*}
S_{\max} & =
\{x\in \ZZ^3 \mid  x_i\leq 3~(i=1,2,3), \ 
x_{1}+x_{2}\leq 5, \ x_{1}+x_{3}\leq 5,
\nonumber \\ & \phantom{= AAAAAA}
 x_{2}+x_{3}\leq 5, \  x_{1}+x_{2}+x_{3} = 6 \}
\nonumber \\ & =
\{x\in \ZZ^3 \mid  1 \leq x_i \leq 3~(i=1,2,3),  \ x_{1}+x_{2}+x_{3} = 6 \}
\end{align*}  
is multimodular by Theorem~\ref{THmmpolydesc}.
Thus the set $S_{\max}$ is both M-convex and multimodular
in spite of the fact that
$S$ itself is \Mnat-convex and not multimodular.
\finbox
\end{example}

\begin{example}  \rm  \label{EXexmnatNOTmmdim4}
Here is an example of an M-convex set that is not multimodular.
Using the set $S$ in \eqref{mnatNOTmmdim3}, consider
\[
\hat S = \{ x \in \ZZ\sp{4} \mid (x_{1},x_{2},x_{3}) \in S, \ x_{4}= 6 -  (x_{1}+x_{2}+x_{3}) \},
\]
which is the M-convex set treated in Example~\ref{EXnorelMnotLmmdim4}.
Here we show that $\hat S$ is not multimodular.
To use Proposition~\ref{PRmultmsetLnatset}, consider
the set
$\hat T = \{ D\sp{-1} x \mid x \in \hat S \}$.
Discrete midpoint convexity fails
for
$y = (2,2,5,6) \in \hat T$ 
(corresponding to $(2,0,3,1) \in \hat S$)
and
$z= (3,4,6,6) \in \hat T$
(corresponding to $(3,1,2,0) \in \hat S$)
with
$(y+z)/2 = (5/2,3,11/2,6)$ and
$\lceil (y+z)/2  \rceil = (3,3,6,6)$,
where 
$w = (3,3,6,6)$
is not contained in $\hat T$
since  the corresponding vector $D w = (3,0,3,0)$ is not in $\hat S$.
\finbox
\end{example}

Finally, we make an observation for the case of three variables.

\begin{proposition} \label{PRmmMdim3}
\quad

\noindent
{\rm (1)}
When $n=3$, every M-convex set is multimodular.

\noindent
{\rm (2)}
When $n=3$, every M-convex function is multimodular.
\end{proposition}
\begin{proof}
(1)
Let $S$ be an M-convex set in $\ZZ\sp{3}$.
Then $S$ is described (cf., \eqref{msetineq}) as
\[
S= \{x\in \ZZ\sp{3} \mid  x(X) \leq \rho(X) \ (X \subset N),  \ x_{1}+x_{2}+x_{3} =  \rho(N) \},
\]
where $N=\{ 1,2,3 \}$. 
Among the subsets of $N$, 
$X= \{ 1, 3\}$ is the only non-consecutive subset.
With the use of the equality constraint $x_{1}+x_{2}+x_{3} =  \rho(N)$,
the corresponding inequality
$x_{1} +x_{3} \leq  \rho(\{ 1,3 \})$
can be rewritten as
$x_{2} \geq \rho(N) - \rho(\{ 1,3 \})$,
which is an inequality constraint for a consecutive subset $\{ 2 \}$.
Then $S$ is multimodular by Theorem~\ref{THmmpolydesc}.

(2)
By Theorem~\ref{THmmf-fp},
the statement (2) for functions follows from (1) for sets.
\end{proof}

\subsubsection{Proof}
\label{SCmnatmmprf}

This section is devoted to the proof of Theorem~\ref{THmmMnat}.

For the proof of the if-part, 
suppose that $S$ is described as \eqref{mmMnatpoly}
with $r$ satisfying \eqref{mmMr0}--\eqref{mmMr2},
where $S$ is nonempty since $\veczero \in S$.
Then $S$ is multimodular by Theorem~\ref{THmmpolydesc}.
To discuss \Mnat-convexity, we define a set function $\rho$ from the given $r$.
We represent a subset $X$ of $N$ as a union of
disjoint consecutive subsets:
\[
X = I(a_{1},b_{1}) \cup I(a_{2},b_{2}) \cup \cdots \cup I(a_{m},b_{m})
\]
with 
$a_{1} \leq b_{1} < a_{2} \leq b_{2} < \cdots < a_{m} \leq b_{m}$,
and define $\rho(X)$ by
\begin{equation} \label{mmMnatpRhofromR}
\rho(X) = r(a_{1},b_{1})  + r(a_{2},b_{2}) + \cdots + r(a_{m},b_{m}),
\end{equation}
where $\rho(\emptyset)= 0$.
Then we have
\begin{equation} \label{mmMnatpRhoI}
 \rho(X) = \rho(I(a_{1},b_{1})) + \rho(I(a_{2},b_{2})) +  \cdots + \rho(I(a_{m},b_{m})).
\end{equation}
For any $x \in S$ and $X \subseteq N$, we have
\[
 x(X) = \sum_{j=1}\sp{m} x(I(a_{j},b_{j})) 
       \leq \sum_{j=1}\sp{m} r(a_{j},b_{j})  = \rho(X),
\]
which shows that inequalities
$x(X)  \leq \rho(X)$
for non-consecutive subsets $X$ may be added to the expression in \eqref{mmMnatpoly}. 
Therefore we have
\begin{equation*}
S =  \{ x \in \ZZ_{+}\sp{n} \mid x(X) \leq \rho(X) \ (X \subseteq N) \}.
\end{equation*}
By the construction, $\rho$ is monotone (nondecreasing) with $\rho(\emptyset)= 0$.
Moreover, $\rho$ is submodular, which is shown in Lemma~\ref{LMmmMnatSubm}
at the end of this section.
Therefore, by \eqref{polymatineq}, 
$S$ is an \Mnat-convex set. 
This completes the proof of the if-part of Theorem~\ref{THmmMnat}.

\medskip

For the proof of the only-if part,  let $S$ be a bounded
\Mnat-convex and multimodular set 
satisfying $\veczero \in S \subseteq \ZZ_{+}\sp{n}$.
By \eqref{polymatineq} for an \Mnat-convex set, we have
\begin{equation*} 
 S= \{ x \in \ZZ_{+}\sp{n} \mid  x(X) \leq \rho(X) \ (X \subseteq N) \}
\end{equation*}
with a nondecreasing integer-valued submodular function
$\rho: 2\sp{N} \to \ZZ$
with $\rho(\emptyset) = 0$, where we may assume that $\rho$ is tight 
in the sense of $\rho(X) = \max \{ x(X) \mid x \in S \}$
$(X \subseteq N)$.
Since $S$ is multimodular, every inequality 
$x(X) \leq \rho(X)$ for a non-consecutive subset $X$
must be redundant by Theorem~\ref{THmmpolydesc}.
Hence we have
\[
 S= \{ x \in \ZZ_{+}\sp{n} \mid 
 x(I) \leq \rho(I) \ \ (\mbox{\rm $I$: consecutive subset of $N$}) \}.
\]
This coincides with the expression \eqref{mmMnatpoly}
for $r(a,b)$ defined by $r(a,b) = \rho(I(a,b))$.
Note that $r(a,b)$ satisfies \eqref{mmMr0}--\eqref{mmMr2}
by the properties \eqref{polymatr1}--\eqref{polymatr3} of $\rho$.
This completes the proof of the only-if-part of Theorem~\ref{THmmMnat}.

\medskip

Finally we show the submodularity of $\rho$ in \eqref{mmMnatpRhofromR}.

\begin{lemma} \label{LMmmMnatSubm}
$\rho$ in \eqref{mmMnatpRhofromR} is submodular.
\end{lemma} 
\begin{proof}
For any subset $Z$ of $N$, we consider submodularity of $\rho$ 
restricted to subsets of $Z$:
\begin{equation} \label{sbmZ}
 \rho(X) + \rho(Y) \geq  \rho(X \cup Y) + \rho(X \cap Y) 
\qquad (X, Y \subseteq Z).
\end{equation}
We prove this by induction on $\alpha=|Z|$.

Obviously, \eqref{sbmZ} is true when $|Z|=1$.
When $|Z|=2$, we have two cases, depending on whether $Z= \{ a,b \}$ is consecutive or not.
If $a+1 =b$, we may assume
$X = \{ a \}$ and $Y = \{ a+1 \}$, for which 
\eqref{sbmZ} follows from \eqref{mmMr2}
since 
$\rho(X) = r(a,a)$,
$\rho(Y) = r(a+1,a+1)$,
$\rho(X \cup Y) = r(a,a+1)$,
 and
$\rho(X \cap Y) = 0$.
If $a+1 < b$, we may assume
$X = \{ a \}$ and $Y = \{ b \}$, for which 
\eqref{sbmZ} holds with equality 
by \eqref{mmMnatpRhoI} and $\rho(X \cap Y) = 0$.

Now assume $\alpha = |Z| \geq 3$.
To prove \eqref{sbmZ} under the induction hypothesis, it suffices to show 
\begin{equation} \label{sbmZloc}
 \rho(Z) + \rho(Z - \{ p, q \}) - \rho(Z - p)- \rho(Z - q) \leq 0
\qquad (p < q; p, q \in Z),
\end{equation}
where $\rho(Z -p)$ means $\rho(Z \setminus \{ p \})$, etc.
We divide the subsequent argument into three cases.
Case (a):
$Z$ is a consecutive set;
Case (b): 
$Z = I_{1} \cup I_{2} \cup \cdots \cup I_{m}$
with consecutive sets $I_{1}, I_{2}, \ldots, I_{m}$
($m \geq 2$), and $p, q \in I_{k}$ for some $k$;
Case (c):
$Z = I_{1} \cup I_{2} \cup \cdots \cup I_{m}$
as in (b), but $p \in I_{k}$ and $q \in I_{l}$ for some $k \ne l$.
It turns out that Case (a) is the essential case.

\medskip

Case (a): 
Let $Z = I(a,b)$.  We have $a \leq p < q \leq b$.
Define
\[
I = \{ i \in Z \mid i < p \},
\quad
J = \{ i \in Z \mid p < i < q \},
\quad
K = \{ i \in Z \mid i > q \}.
\]
We have
\[
 \rho(Z) \leq \rho(Z - a) + \rho(Z - b) - \rho(Z - \{ a, b \})
\]
by \eqref{mmMr2}, whereas \eqref{mmMnatpRhoI} gives
\begin{align*}
 \rho(Z - \{ p, q \}) & =  \rho(I) + \rho(J) + \rho(K),
\\
\rho(Z - p ) & =  \rho(I) + \rho(J q K),
\\
\rho(Z - q)  & =   \rho(I p J) + \rho(K),
\end{align*}
where
$\rho(J  q  K)$ means $\rho(J \cup \{ q \} \cup K)$, etc.
Substituting these into the left-hand side of \eqref{sbmZloc}, we obtain
\begin{align}
& \mbox{LHS of \eqref{sbmZloc}}
\nonumber \\ & 
 \leq 
  \rho(Z - a) + \rho(Z - b) - \rho(Z - \{ a, b \}) 
 + \rho(J) - \rho(J q K) - \rho(I p J) 
\nonumber \\ & =
 [ \rho(Z - a) + \rho(J q K - b) - \rho(Z - \{ a, b \}) - \rho(J q K) ]
\nonumber \\ & \phantom{= } \ {} 
 + [ \rho(Z - b) + \rho(J) - \rho(I p J) - \rho(J q K - b) ]
\nonumber \\ & = [ \rho(X' \cup Y') + \rho(X' \cap Y') - \rho(X') - \rho(Y') ]
\nonumber \\ & \phantom{= } \ {} 
{} + [ \rho(X'' \cup Y'') + \rho(X'' \cap Y'') - \rho(X'') - \rho(Y'') ] ,
\label{sbmZlocLHS}
\end{align}
where
$X' := Z - \{ a, b \}$ and $Y':= J q K$, 
for which
$Z - a = X' \cup Y'$ and $J q K - b = X' \cap Y'$,
and
$X'':=I p J$ and $Y'':=J q K - b$ $(= X' \cap Y')$, 
for which
$Z - b =  X'' \cup Y''$ and $J = X'' \cap Y''$.
Since
$|Z - a| = |Z - b| = \alpha -1$,
\eqref{sbmZlocLHS}
is non-positive ($\leq 0$) by the induction hypothesis 
\eqref{sbmZ} for $\alpha -1$.

Case (b): 
By \eqref{mmMnatpRhoI} we have
\begin{align*}
 &  \rho(Z) =
\rho(I_{k}) +
\sum_{j \ne k} \rho(I_{j}) ,
\quad
\rho(Z - \{ p, q \})  =
\rho(I_{k} - \{ p, q \} ) +
\sum_{j \ne k} \rho(I_{j})  ,
\\
& \rho(Z - p) =
\rho(I_{k} - p) +
\sum_{j \ne k} \rho(I_{j}) ,
\quad
 \rho(Z - q) =
\rho(I_{k} - q) +
\sum_{j \ne k} \rho(I_{j})  .
\end{align*}
Therefore,
\[
\mbox{LHS of \eqref{sbmZloc}}
=   \rho(I_{k}) + \rho(I_{k} - \{ p, q \})  - \rho(I_{k}-p) - \rho(I_{k}-q),
\]
which is non-positive ($\leq 0$) by the induction hypothesis 
\eqref{sbmZ} since 
$|I_{k}| \leq \alpha -1$.

Case (c): 
By \eqref{mmMnatpRhoI} we have
\begin{align*}
 &  \rho(Z) =
\rho(I_{k}) + \rho(I_{l}) +
\sum_{j \ne k,l} \rho(I_{j}) ,
\quad
\rho(Z - \{ p, q \})  =
\rho(I_{k} - p ) + \rho(I_{l} - q) +
\sum_{j \ne k,l} \rho(I_{j})  ,
\\
& \rho(Z - p) =
\rho(I_{k} - p) + \rho(I_{l}) +
\sum_{j \ne k,l} \rho(I_{j}) ,
\quad
 \rho(Z - q) =
\rho(I_{k}) + \rho(I_{l} - q) +
\sum_{j \ne k,l} \rho(I_{j})  .
\end{align*}
Therefore,
$\mbox{LHS of \eqref{sbmZloc}}=0$.

Thus we have shown \eqref{sbmZloc} in all cases (a), (b), and (c).
Therefore, \eqref{sbmZ} holds when $|Z|=\alpha$.
\end{proof}

\appendix

\section{Definitions of Discrete Convex Functions}
\label{SCdiscfndef}

This section offers definitions of various concepts of discrete convex functions
such as 
integrally convex,  
{\rm L}-convex, and
{\rm M}-convex functions.
The definition of multimodular functions is given in Section~\ref{SCmmfnDef}.
We consider functions defined on integer lattice points, 
$f: \ZZ\sp{n} \to \RR \cup \{ +\infty \}$,
where the function may possibly take $+\infty$ but 
it is assumed that the effective domain, 
$\dom f =\{ x \mid f(x) < + \infty \}$, is nonempty.

\subsection{Separable convexity}

For integer vectors 
$a \in (\ZZ \cup \{ -\infty \})\sp{n}$ and 
$b \in (\ZZ \cup \{ +\infty \})\sp{n}$ 
with $a \leq b$,
$[a,b]_{\ZZ}$ denotes the box of integers  
(discrete rectangle, integer interval)
between $a$ and $b$,
i.e.,
\begin{equation} \label{intboxdef}
[a,b]_{\ZZ} = \{ x \in \ZZ\sp{n} \mid a_{i} \leq x_{i} \leq b_{i} \ (i=1,2,\ldots,n)  \}.
\end{equation}

A function
$f: \ZZ^{n} \to \RR \cup \{ +\infty \}$
in $x=(x_{1}, x_{2}, \ldots,x_{n}) \in \ZZ^{n}$
is called  {\em separable convex}
if it can be represented as
\begin{equation}  \label{sepfndef}
f(x) = \varphi_{1}(x_{1}) + \varphi_{2}(x_{2}) + \cdots + \varphi_{n}(x_{n})
\end{equation}
with univariate functions
$\varphi_{i}: \ZZ \to \RR \cup \{ +\infty \}$ satisfying 
\begin{equation}  \label{univarconvdef}
\varphi_{i}(t-1) + \varphi_{i}(t+1) \geq 2 \varphi_{i}(t)
\qquad (t \in \ZZ),
\end{equation}
where $\dom \varphi_{i}$ is an interval of integers.

\subsection{Integral convexity}

For $x \in \RR^{n}$ the integral neighborhood of $x$ is defined as 
\begin{equation}  \label{intneighbordef}
N(x) = \{ z \in \ZZ^{n} \mid | x_{i} - z_{i} | < 1 \ (i=1,2,\ldots,n)  \}.
\end{equation}
It is noted that 
strict inequality ``\,$<$\,'' is used in this definition
and hence $N(x)$ admits an alternative expression
\begin{equation}  \label{intneighbordeffloorceil}
N(x) = \{ z \in \ZZ\sp{n} \mid
\lfloor x_{i} \rfloor \leq  z_{i} \leq \lceil x_{i} \rceil  \ \ (i=1,2,\ldots, n) \} .
\end{equation}
For a set $S \subseteq \ZZ^{n}$
and $x \in \RR^{n}$
we call the convex hull of $S \cap N(x)$ 
the {\em local convex hull} of $S$ at $x$.
A nonempty set $S \subseteq \ZZ^{n}$ is said to be 
{\em integrally convex} if
the union of the local convex hulls $\overline{S \cap N(x)}$ over $x \in \RR^{n}$ 
is convex \cite{Mdcasiam}.
This is equivalent to saying that,
for any $x \in \RR^{n}$, 
$x \in \overline{S} $ implies $x \in  \overline{S \cap N(x)}$.

For a function
$f: \ZZ^{n} \to \RR \cup \{ +\infty  \}$
the {\em local convex extension} 
$\tilde{f}: \RR^{n} \to \RR \cup \{ +\infty \}$
of $f$ is defined 
as the union of all convex envelopes of $f$ on $N(x)$.  That is,
\begin{equation} \label{fnconvclosureloc2}
 \tilde f(x) = 
  \min\{ \sum_{y \in N(x)} \lambda_{y} f(y) \mid
      \sum_{y \in N(x)} \lambda_{y} y = x,  \ 
  (\lambda_{y})  \in \Lambda(x) \}
\quad (x \in \RR^{n}) ,
\end{equation} 
where $\Lambda(x)$ denotes the set of coefficients for convex combinations indexed by $N(x)$:
\[ 
  \Lambda(x) = \{ (\lambda_{y} \mid y \in N(x) ) \mid 
      \sum_{y \in N(x)} \lambda_{y} = 1, 
      \lambda_{y} \geq 0 \ \ \mbox{for all } \   y \in N(x)  \} .
\] 
If $\tilde f$ is convex on $\RR^{n}$,
then $f$ is said to be {\em integrally convex}
\cite{FT90}.
The effective domain of an integrally convex function is an integrally convex set.
A set $S \subseteq \ZZ\sp{n}$ is integrally convex if and only if its indicator function
$\delta_{S}: \ZZ\sp{n} \to \{ 0, +\infty \}$
is an integrally convex function.

Integral convexity of a function can be characterized by a local condition
under the assumption that the effective domain is an integrally convex set.

\begin{theorem}[\cite{FT90,MMTT19proxIC}]
\label{THfavtarProp33}
Let $f: \ZZ^{n} \to \RR \cup \{ +\infty  \}$
be a function with an integrally convex effective domain.
Then the following properties are equivalent:


{\rm (a)}
$f$ is integrally convex.

{\rm (b)}
For every $x, y \in \ZZ\sp{n}$ with $\| x - y \|_{\infty} =2$  we have \ 
\begin{equation}  \label{intcnvconddist2}
\tilde{f}\, \bigg(\frac{x + y}{2} \bigg) 
\leq \frac{1}{2} (f(x) + f(y)).
\end{equation}
\vspace{-1.7\baselineskip} \\
\finbox
\end{theorem}

The reader is referred to 
\cite{FT90,MM19projcnvl,MMTT19proxIC,MT20subgrIC,MT21ICfenc,MT22ICsurv},
\cite[Section~3.4]{Mdcasiam}, and 
\cite[Section~13]{Mdcaeco16}
for more about integral convexity,

\subsection{L-convexity}
\label{SCappLconv}

L- and \Lnat-convex functions form  major classes of discrete convex functions
\cite[Chapter~7]{Mdcasiam}.
The concept of \Lnat-convex functions
was introduced in \cite{FM00}
as an equivalent variant of 
{\rm L}-convex functions introduced earlier in \cite{Mdca98}.

\subsubsection{\Lnat-convex functions}

A nonempty set $S \subseteq  \ZZ\sp{n}$ is called {\em \Lnat-convex} if
\begin{equation} \label{midptcnvset}
 x, y \in S
\ \Longrightarrow \
\left\lceil \frac{x+y}{2} \right\rceil ,
\left\lfloor \frac{x+y}{2} \right\rfloor  \in S ,
\end{equation}
where, for $t \in \RR$ in general, 
$\left\lceil  t   \right\rceil$ 
denotes the smallest integer not smaller than $t$
(rounding-up to the nearest integer)
and $\left\lfloor  t  \right\rfloor$
the largest integer not larger than $t$
(rounding-down to the nearest integer),
and this operation is extended to a vector
by componentwise applications.
The property \eqref{midptcnvset} is called {\em discrete midpoint convexity}
(in the original sense of the word).

A function $f : \ZZ\sp{n} \to \RR \cup \{ +\infty \}$ with $\dom f \not= \emptyset$
is said to be {\em L$\sp{\natural}$-convex}
if it satisfies a quantitative version of discrete midpoint convexity,
i.e., if 
\begin{equation} \label{midptcnvfn}
 f(x) + f(y) \geq
   f \left(\left\lceil \frac{x+y}{2} \right\rceil\right) 
  + f \left(\left\lfloor \frac{x+y}{2} \right\rfloor\right) 
\end{equation}
holds for all $x, y \in \ZZ\sp{n}$.
The effective domain of an \Lnat-convex function is an \Lnat-convex set.
A set $S$ is \Lnat-convex if and only if its indicator function
$\delta_{S}$ is an \Lnat-convex function.

It is known \cite[Section~7.1]{Mdcasiam} that
${\rm L}^{\natural}$-convex functions can be characterized 
by several different conditions,
stated in Theorem~\ref{THlnatfncond} below.
The condition (b) in Theorem~\ref{THlnatfncond}
imposes discrete midpoint convexity \eqref{midptcnvfn} 
for all points $x,y$ at $\ell_\infty$-distance 1 or 2.
The condition (c) refers to 
{\em submodularity}, which means that
\begin{equation} \label{submfn}
f(x) + f(y) \geq f(x \vee y) + f(x \wedge y)
\end{equation}
holds for all $x, y \in \ZZ\sp{n}$, 
where
$x \vee y$ and $x \wedge y$ denote,
respectively, the vectors of componentwise maximum and minimum of $x$ and $y$;
see \eqref{veewedgedef}.
The condition (d) refers to 
a generalization of submodularity called
{\em translation-submodularity}, which means that
\begin{equation} \label{trsubmfn}
  f(x) + f(y) \geq f((x - \mu {\bf 1}) \vee y)
                 + f(x \wedge (y + \mu {\bf 1}))
\end{equation}
holds for all $x, y \in \ZZ^{n}$ and nonnegative integers $\mu$,
where $\bm{1}=(1,1,\ldots, 1)$.
The condition (e) refers to the condition%
\footnote{
This condition \eqref{lnatfnAPR} is labeled as  (L$\sp{\natural}$-APR[$\ZZ$]) 
   in \cite[Section~7.2]{Mdcasiam}.
} 
 that, for any $x, y \in \ZZ\sp{n}$ with $\suppp(x-y)\not= \emptyset$, 
the inequality
\begin{equation} \label{lnatfnAPR}
 f(x) + f(y) \geq f(x - \unitvec{A}) + f(y + \unitvec{A}) 
\end{equation}
holds with $\displaystyle A = \argmax_{i} \{ x_{i} - y_{i} \}$,
where 
$\suppp(x-y) = \{ i \mid x_{i} >  y_{i} \}$ and
$\unitvec{A}$ denotes the characteristic vector of $A$.
The condition (f) refers to submodularity of
the function 
\begin{equation}\label{lfnlnatfnrelation}
 \tilde f(x_{0},x) = f(x - x_{0} \vecone)
 \qquad ( x_{0} \in \ZZ, x \in \ZZ\sp{n})
\end{equation}
in $n+1$ variables associated with the given function $f$.

\begin{theorem}[{\cite[Theorem~2.2]{Msurvop21}}] \label{THlnatfncond}
For $f: \ZZ^{n} \to \RR \cup \{ +\infty \}$,
the following conditions, {\rm (a)} to {\rm (f)}, are equivalent:


{\rm (a)}
$f$ is an \Lnat-convex function, that is,
it satisfies 
discrete midpoint inequality
\eqref{midptcnvfn} for all $ x, y \in \ZZ^{n}$.

{\rm (b)}
$\dom f$ is an \Lnat-convex set,
and $f$ satisfies
discrete midpoint inequality
\eqref{midptcnvfn} for all $ x, y \in \ZZ^{n}$
with $\| x-y \|_{\infty} \leq 2$.

{\rm (c)}
$f$ is integrally convex and submodular \eqref{submfn}.

{\rm (d)}
$f$ satisfies translation-submodularity \eqref{trsubmfn}
for all nonnegative $\mu \in \ZZ$.

{\rm (e)}
$f$ satisfies the condition \eqref{lnatfnAPR}.

{\rm (f)}
$\tilde f$ in \eqref{lfnlnatfnrelation} is submodular \eqref{submfn}.
\finbox
\end{theorem}

For a set $S \subseteq  \ZZ\sp{n}$ we consider conditions
\begin{align} 
 x, y \in S
& \ \Longrightarrow \
 x \vee y, \ x \wedge y \in S ,
\label{submset}
\\
 x, y \in S
& \ \Longrightarrow \
 (x - \mu {\bf 1}) \vee y, \  x \wedge (y + \mu {\bf 1}) \in S ,
\label{trsubmset}
\\
x,y \in S, \  \suppp(x-y)\not= \emptyset 
& \ \Longrightarrow \
x - \unitvec{A}, \  y + \unitvec{A} \  \in S
\mbox{ for } A = \argmax_{i} \{ x_{i} - y_{i} \} .
\label{lnatsetAPR}
\end{align}
The first condition \eqref{submset}, meaning that $S$ is a sublattice of $\ZZ\sp{n}$,
corresponds to submodularity \eqref{submfn},
whereas
\eqref{trsubmset} and \eqref{lnatsetAPR} correspond to 
\eqref{trsubmfn} and \eqref{lnatfnAPR}, respectively.

\begin{theorem}[{\cite[Proposition~2.3]{Msurvop21}}]  \label{THlnatset}
For a nonempty set $S \subseteq  \ZZ\sp{n}$, 
the following conditions, {\rm (a)} to {\rm (d)}, are equivalent:


{\rm (a)}
$S$ is an \Lnat-convex set, that is,
it satisfies \eqref{midptcnvset}.

{\rm (b)}
$S$ is an integrally convex set that
satisfies \eqref{submset}.

{\rm (c)}
$S$ satisfies \eqref{trsubmset} for all nonnegative $\mu \in \ZZ$.

{\rm (d)}
$S$ satisfies \eqref{lnatsetAPR}.
\finbox
\end{theorem}

The following polyhedral description of an \Lnat-convex set is known
\cite[Section~5.5]{Mdcasiam}.

\begin{theorem} \label{THlnatpolydesc}
A set $S \subseteq \ZZ\sp{n}$ is \Lnat-convex if and only if 
$S = \overline{S} \cap \ZZ\sp{n}$ and
its convex hull $\overline{S}$ can be represented as
\begin{equation}  \label{SnatfromGamma}
\overline{S} = \{ x \in \RR\sp{n} \mid 
\alpha_{i} \leq  x_{i} \leq \beta_{i} \ \ (i\in N),  \ 
x_{j} - x_{i} \leq \gamma_{ij} \ \ (i,j \in N) \}
\end{equation}
for some $\alpha_{i} \in \ZZ \cup \{ -\infty \}$,
$\beta_{i} \in \ZZ \cup \{ +\infty \}$,
and $\gamma_{ij} \in \ZZ \cup \{ +\infty \}$
$(i,j \in N )$
with $\gamma_{ii} =0$ $(i \in N )$
such that 
$\tilde \gamma_{ij}$ defined
for $i,j \in N \cup \{ 0 \}$
by 
\begin{equation}  \label{tildeGamma}
\tilde \gamma_{00} =0, \qquad
\tilde \gamma_{ij} = \gamma_{ij}, \quad
\tilde \gamma_{i0} = -\alpha_{i}, \quad
\tilde \gamma_{0j} = \beta_{j}  \qquad
(i,j \in N)
\end{equation}
satisfies the triangle inequality:
\begin{equation}  \label{gammatriangleLnat}
\tilde \gamma_{ij} + \tilde \gamma_{jk} \geq \tilde \gamma_{ik}
\qquad (i,j,k \in N \cup \{ 0 \}).
\end{equation}
Such $\alpha_{i}$, $\beta_{i}$, $\gamma_{ij}$ are determined from $S$ by
\begin{align} 
&\alpha_{i} = \min\{ x_{i} \mid x \in S \}, \quad
\beta_{i} =  \max\{ x_{i} \mid x \in S \}
\qquad (i \in N),
 \label{alphabetaFromSnat}
\\ &
\gamma_{ij} =  \max\{ x_{j} - x_{i} \mid x \in S \} 
\qquad (i,j \in N).
 \label{gammaFromSnat}
\end{align}
\vspace{-1.3\baselineskip}\\
\finbox
\end{theorem}

\begin{remark} \rm \label{RMpolydescL}
Here are additional remarks about the polyhedral descriptions in Theorem~\ref{THlnatpolydesc}.
\begin{itemize}
\item
The correspondence between $S$ and integer-valued $(\alpha, \beta, \gamma)$ 
with \eqref{gammatriangleLnat}
is bijective (one-to-one and onto)
through \eqref{SnatfromGamma}, \eqref{alphabetaFromSnat}, and \eqref{gammaFromSnat}.

\item
For any integer-valued $(\alpha, \beta, \gamma)$ 
(independent of the triangle inequality),
$S$ in \eqref{SnatfromGamma}
is an \Lnat-convex set if $S \ne \emptyset$.
We have $S \ne \emptyset$ if and only if there exists no negative cycle 
with respect to $\tilde \gamma$,
where a negative cycle means a set of indices $i_{1}, i_{2}, \ldots, i_{m}$ 
such that
$\tilde \gamma_{i_{1}i_{2}} + \tilde \gamma_{i_{2}i_{3}}+ \cdots
 + \tilde \gamma_{i_{m-1}i_{m}} + \tilde \gamma_{i_{m}i_{1}} < 0$.
\finbox
\end{itemize}
\end{remark}

\subsubsection{L-convex functions}

A function 
$f(x_{1},x_{2}, \ldots, x_{n} )$
with $\dom f \not= \emptyset$ is called {\em L-convex}
if it is submodular \eqref{submfn}
 and there exists
$r \in \RR$ such that 
\begin{equation}\label{shiftlfnZ}
f(x + \mu \vecone) = f(x) + \mu  r
\end{equation}
for all $x  \in \ZZ\sp{n}$ and $\mu \in \ZZ$.
If $f$ is {\rm L}-convex,
the function
$g(x_{2}, \ldots, x_{n} ) := f(0, x_{2}, \ldots, x_{n} )$
is an \Lnat-convex function, and any \Lnat-convex function 
arises in this way.
The function $\tilde f$ in \eqref{lfnlnatfnrelation} derived from an \Lnat-convex function $f$  
is an {\rm L}-convex function
with $\tilde f(x_{0}+ \mu , x + \mu  \vecone) = \tilde f(x_{0},x)$,
and we have $f(x)  =  \tilde f(0,x)$.

\begin{theorem}[{\cite[Theorem~2.4]{Msurvop21}}] \label{THlfncond}
For $f: \ZZ^{n} \to \RR \cup \{ +\infty \}$,
the following conditions, {\rm (a)} to {\rm (c)}, are equivalent:


{\rm (a)}
$f$ is an L-convex function, that is, it satisfies \eqref{submfn} 
and \eqref{shiftlfnZ} for some $r \in \RR$.

{\rm (b)}
$f$ is an \Lnat-convex function that satisfies \eqref{shiftlfnZ} for some $r \in \RR$.

{\rm (c)}
$f$ satisfies translation-submodularity \eqref{trsubmfn}
for all $\mu \in \ZZ$ 
(including $\mu < 0$).
\finbox
\end{theorem}

A nonempty set $S$ is called {\em L-convex} if its indicator function
$\delta_{S}$
is an L-convex function.
This means that $S$ is L-convex if and only if 
it satisfies 
\eqref{submset} and
\begin{equation} \label{invarone}
 x  \in S, \  \mu \in \ZZ  \ \Longrightarrow \   \  x +  \mu \vecone \in S .
\end{equation}
This property is sometimes called the translation invariance in the direction of $\vecone$.
The effective domain of an L-convex function is an L-convex set.

The following theorem gives equivalent conditions for a set to be L-convex.

\begin{theorem}[{\cite[Proposition~2.5]{Msurvop21}}] \label{THlset}
For a nonempty set $S \subseteq  \ZZ\sp{n}$, 
the following conditions, {\rm (a)} to {\rm (c)}, are equivalent:


{\rm (a)}
$S$ is an L-convex set, that is, it satisfies \eqref{submset} and \eqref{invarone}.

{\rm (b)}
$S$ is an \Lnat-convex set that satisfies \eqref{invarone}.

{\rm (c)}
$S$ satisfies \eqref{trsubmset} for all $\mu \in \ZZ$
(including $\mu < 0$).
\finbox
\end{theorem}

The following theorem reduces the concept of L-convex functions to that of L-convex sets.

\begin{theorem}
[{\cite[Section~16.2]{Fuj05book}}, {\cite[Theorem~7.17]{Mdcasiam}}] 
\label{THlf-fp}
Under the assumption \eqref{ABCfunAssm},
a function $f:\ZZ\sp{n} \to \Rinf$ 
is L-convex (resp., \Lnat-convex) 
if and only if,
for any vector $p \in \RR\sp{n}$, $\argmin f[-p]$ is an L-convex (resp., \Lnat-convex) set
or an empty set.
\finbox
\end{theorem}

\subsubsection{Discrete midpoint convex functions}

A nonempty set $S \subseteq  \ZZ\sp{n}$ is said to be {\em discrete midpoint convex} if
\begin{equation} \label{dirintcnvsetdef}
 x, y \in S, \ \| x - y \|_{\infty} \geq 2
\ \Longrightarrow \
\left\lceil \frac{x+y}{2} \right\rceil ,
\left\lfloor \frac{x+y}{2} \right\rfloor  \in S.
\end{equation}
This condition is  weaker than the defining condition \eqref{midptcnvset}
for an \Lnat-convex set,
and hence every \Lnat-convex set
is a discrete midpoint convex set.

A function $f: \ZZ\sp{n} \to \RR \cup \{ +\infty \}$
is called {\em  globally discrete midpoint convex} 
\cite{MMTT20dmc}
if the discrete midpoint convexity 
\begin{equation} \label{midptcnvfn2}
 f(x) + f(y) \geq
   f \left(\left\lceil \frac{x+y}{2} \right\rceil\right) 
  + f \left(\left\lfloor \frac{x+y}{2} \right\rfloor\right) 
\end{equation}
is satisfied by every pair $(x, y) \in \ZZ\sp{n} \times \ZZ\sp{n}$
with $\| x - y \|_{\infty} \geq 2$.
The effective domain of a
globally discrete midpoint convex function
is necessarily a discrete midpoint convex set.
Obviously, every \Lnat-convex function is globally discrete midpoint convex.
We sometimes abbreviate ``discrete midpoint convex(ity)'' to ``d.m.c.''

We define notation 
$\tilde \mu(a,b)$
for a pair of integers $(a, b)$ by
\begin{equation} \label{ddmcrounddefInt}
\tilde \mu(a,b) = 
 \begin{cases} 
\left\lceil (a+b)/{2} \right\rceil
 & \ (a \geq b) , \\
\left\lfloor (a+b)/{2} \right\rfloor
 & \ (a \leq b) ,
 \end{cases}
\end{equation}
and extend this notation 
to a pair of integer vectors $(x, y)$ as
\begin{equation} \label{ddmcrounddefVec}
\tilde \mu(x,y) =  
\left( \tilde \mu(x_{1},y_{1}), \tilde \mu(x_{2},y_{2}), \ldots, \tilde \mu(x_{n},y_{n})  \right).
\end{equation}
A function $f: \ZZ\sp{n} \to \RR \cup \{ +\infty \}$
is called {\em  directed discrete midpoint convex} 
\cite{TT21ddmc}
if the inequality
\begin{equation} \label{ddmcfndef}
 f(x) + f(y) \geq 
 f( \tilde \mu(x,y)  ) +    f( \tilde \mu(y,x)  ) 
\end{equation}
is satisfied by every pair $(x, y) \in \ZZ\sp{n} \times \ZZ\sp{n}$.
A nonempty set $S$ is called {\em directed discrete midpoint convex} 
if its indicator function
$\delta_{S}$ is a directed discrete midpoint convex function.
\Lnat-convex functions are directed discrete midpoint convex functions,
and \Lnat-convex sets are directed discrete midpoint convex sets
(see \eqref{fnfamddmc} in Proposition \ref{PRfnclassInc}).

\subsubsection{\LL-convex functions}
\label{SCl2fn}

A nonempty set $S$ $\subseteq \ZZ\sp{n}$
is called 
{\em \LL-convex}  
(resp., {\em \LLnat-convex})  
if it can be represented as the Minkowski sum
(vector addition)
of two L-convex (resp., \Lnat-convex) sets
\cite[Section~5.5]{Mdcasiam}.
That is,
\[
 S = \{ x + y \mid x \in S_{1} ,  y \in S_{2} \},
\]
where
$S_{1}$ and $S_{2}$ are L-convex (resp., \Lnat-convex) sets.
An \LLnat-convex set 
is the intersection of an 
\LL-convex set 
with a coordinate hyperplane.
That is, for L-convex sets, 
the operations of the Minkowski addition and 
the restriction to a coordinate hyperplane
commute with each other.
This fact is stated in \cite[p.129]{Mdcasiam} without a proof,
and a proof can be found in \cite[Section~2.3]{MM21l2ineq}.
The polyhedral description of \LL-convex and \LLnat-convex sets
is given in \cite{MM21l2ineq}. 
An \LL-convex set has the property of translation invariance 
in \eqref{invarone}.

A function 
$f: \ZZ\sp{n} \to \RR \cup \{ +\infty \}$
is said to be 
{\em \LL-convex}
if it can be represented as the (integral) infimal convolution 
$f_{1} \Box f_{2}$  of 
two L-convex functions
$f_{1}$ and $f_{2}$, that is,
if
\[
f(x) =  (f_{1} \Box f_{2})(x) = \inf\{ f_{1}(y)+ f_{2}(z) 
        \mid x = y+z; \  y,z \in \ZZ\sp{n}  \}
\qquad (x \in \ZZ\sp{n}).
\]
It is known \cite{Tam03dec}
(see also \cite[Note 8.37]{Mdcasiam})
that the infimum
is always attained as long as it is finite. 
The (integer) infimal convolution of two \Lnat-convex functions
is called an {\em \LLnat-convex function}.
A nonempty set $S$ is \LL-convex (resp., \LLnat-convex)
if and only if its indicator function
$\delta_{S}$ is \LL-convex (resp., \LLnat-convex).

\subsection{M-convexity}
\label{SCappMconv}

{\rm M}- and \Mnat-convex functions form major classes of discrete convex functions
\cite[Chapter~6]{Mdcasiam}.
The concept of \Mnat-convex functions
was introduced in \cite{MS99gp}
as an equivalent variant of 
{\rm M}-convex functions introduced earlier in \cite{Mstein96}.

\subsubsection{\Mnat-convex functions}
\label{SCmnatfn}

For two vectors $x, y \in \ZZ\sp{n}$ we use notations
\begin{equation*} 
\suppp(x-y) = \{ i \mid x_{i} >  y_{i} \},
\qquad
\suppm(x-y) = \{ i \mid x_{i} <  y_{i} \} .
\end{equation*}
We say that a function
$f: \ZZ\sp{n} \to \RR \cup \{ +\infty \}$
with $\dom f \not= \emptyset$
is {\em M$\sp{\natural}$-convex}, if,
for any $x, y \in \ZZ\sp{n}$ and $i \in \suppp(x-y)$, 
we have (i)
\begin{equation}  \label{mconvex1Z}
f(x) + f(y)  \geq  f(x -\unitvec{i}) + f(y+\unitvec{i})
\end{equation}
or (ii) there exists some $j \in \suppm(x-y)$ such that
\begin{equation}  \label{mconvex2Z}
f(x) + f(y)   \geq 
 f(x-\unitvec{i}+\unitvec{j}) + f(y+\unitvec{i}-\unitvec{j}) .
\end{equation}
This property is referred to as the {\em exchange property}.
A more compact expression of the exchange property is as follows:
\begin{description}
\item[\Mnvexb]
 For any $x, y \in \ZZ\sp{n}$ and $i \in \suppp(x-y)$, we have
\begin{equation} \label{mnconvexc2Z}
f(x) + f(y)   \geq 
\min_{j \in \suppm(x - y) \cup \{ 0 \}} 
 \{ f(x - \unitvec{i} + \unitvec{j}) + f(y + \unitvec{i} - \unitvec{j}) \},
\end{equation}
\end{description}
where $\unitvec{0}=\veczero$ (zero vector).
In the above statement we may change
``For any $x, y \in \ZZ\sp{n}$\,'' to ``For any $x, y \in \dom f$\,''
 since if $x \not\in \dom f$ or $y \not\in \dom f$,
the inequality \eqref{mnconvexc2Z} trivially holds with $f(x) + f(y)  = +\infty$.

M$\sp{\natural}$-convex functions can be characterized by 
a number of different exchange properties including a local exchange property
under the assumption that 
function $f$ is (effectively) defined on an M$\sp{\natural}$-convex set.
See 
\cite[Theorem~6.8]{ST15jorsj},
\cite[Chapters 4 and 6]{Mdcasiam},
\cite[Section~4]{Mdcaeco16}, and
\cite{MS18mnataxiom}
for detailed discussion about the exchange properties.

It follows from \Mnvex that the effective domain $S = \dom f$
of an M$\sp{\natural}$-convex function $f$
has the following exchange property:
\begin{description}
\item[\Bnvexb] 
For any $x, y \in S$ and $i \in \suppp(x-y)$, we have
(i)
$x -\unitvec{i} \in S$ and $y+\unitvec{i} \in S$
\ or \  
\\
(ii) there exists some $j \in \suppm(x-y)$ such that
$x-\unitvec{i}+\unitvec{j}  \in S$ and $y+\unitvec{i}-\unitvec{j} \in S$.
\end{description}
A nonempty set $S \subseteq \ZZ\sp{n}$ having this property 
is called an {\em \Mnat-convex set}.

An \Mnat-convex set
is nothing but the set of integer points in an integral generalized polymatroid
\cite[Section~4.7]{Mdcasiam}.
In particular, an integral polymatroid
can be identified with a bounded \Mnat-convex set 
containing $\veczero$ and consisting of nonnegative vectors.
As is well known \cite[Section~2.2]{Fuj05book}, 
the set of integer points of an integral polymatroid
is described as
\begin{equation} \label{polymatineq}
 S= \{ x \in \ZZ\sp{n} \mid x \geq \veczero, \  x(X) \leq \rho(X) \ (X \subseteq N) \}
\end{equation}
with a nondecreasing integer-valued submodular function $\rho$.
To be more specific, the set function $\rho: 2\sp{N} \to \ZZ$ should satisfy
\begin{align}
&
\rho(\emptyset) = 0,
\label{polymatr1}
\\ &
X \subseteq Y \  \Longrightarrow \ 
\rho(X) \leq \rho(Y), 
\label{polymatr2}
\\ &
\rho(X) + \rho(Y) \geq
\rho(X \cup Y) + \rho(X \cap Y)
\quad (X,Y \subseteq N).
\label{polymatr3}
\end{align}
Moreover, the convex hull $\overline{S}$ of $S$ is described similarly as
$\overline{S} = 
\{ x \in \RR\sp{n} \mid x \geq \veczero, \  x(X) \leq \rho(X) \ (X \subseteq N) \}$
and we can determine $\rho$ from $S$ by
$\rho(X) = \max \{ x(X) \mid x \in S \}$
$(X \subseteq N)$.

\subsubsection{M-convex functions}
\label{SCmfn}

If a set $S \subseteq \ZZ\sp{n}$ lies on a hyperplane with a constant component sum
(i.e., $x(N) = y(N)$ for all  $x, y \in S$),
the exchange property \Bnvex
takes a simpler form
(without the possibility of the first case (i)): 
\begin{description}
\item[\Bvexb] 
For any $x, y \in S$ and $i \in \suppp(x-y)$, 
there exists some $j \in \suppm(x-y)$ such that
$x-\unitvec{i}+\unitvec{j}  \in S$ and $y+\unitvec{i}-\unitvec{j} \in S$.
\end{description}
A nonempty set $S \subseteq \ZZ\sp{n}$ having this exchange property 
is called an {\em M-convex set},
which is an alias for the set of integer points in an integral base polyhedron.  
An M-convex set $S$ contained in the nonnegative orthant $\ZZ_{+}\sp{n}$ 
can be described as
\begin{equation} \label{msetineq}
 S= \{ x \in \ZZ\sp{n} \mid  x(X) \leq \rho(X) \ (X \subset N), \  x(N) = \rho(N)   \}
\end{equation}
with a set function $\rho: 2\sp{N} \to \ZZ$
satisfying \eqref{polymatr1}, \eqref{polymatr2}, and \eqref{polymatr3}.

An M$\sp{\natural}$-convex function 
whose effective domain is an M-convex set
is called an {\em M-convex function}
\cite{Mstein96,Mdca98,Mdcasiam}.
In other words,  a function 
$f: \ZZ\sp{n} \to \RR \cup \{ +\infty \}$
is M-convex
if and only if it satisfies
the exchange property:
\begin{description}
\item[\Mvexb]
 For any $x, y \in \dom f$  and $i \in \suppp(x-y)$, 
there exists
$j \in \suppm(x-y)$ such that
\eqref{mconvex2Z} holds.
\end{description}
M-convex functions can be characterized by a local exchange property
under the assumption 
that function $f$ is (effectively) defined on an M-convex set.
See \cite[Section~6.2]{Mdcasiam}.

M-convex functions and
M$\sp{\natural}$-convex functions are equivalent concepts,
in that M$\sp{\natural}$-convex functions in $n$ variables 
can be obtained as projections of M-convex functions in $n+1$ variables.
More formally, let ``$0$'' denote a new element not in $N$ and
$\tilde{N} = \{0\} \cup N =\{ 0,1,\ldots,n  \}$.
A function $f : \ZZ\sp{n} \to \RR \cup \{ +\infty \}$ is M$\sp{\natural}$-convex
if and only if  the function 
$\tilde{f} : \ZZ\sp{n+1} \to \RR \cup \{ +\infty \}$ 
defined by
\begin{equation} \label{mfnmnatfnrelationvex}
  \tilde{f}(x_{0},x) = \left\{ \begin{array}{ll}
      f(x)    & \mbox{ if $x_{0} = {-}x(N)$} \\
      +\infty & \mbox{ otherwise}
    \end{array}\right.
 \qquad ( x_{0} \in \ZZ, x \in \ZZ\sp{n})
\end{equation}
is an M-convex function.

The following theorem reduces the concept of M-convex functions to that of M-convex sets.

\begin{theorem}
[{\cite[Section~17]{Fuj05book}}, {\cite[Theorem~6.30]{Mdcasiam}}]
\label{THmf-fp}
Under the assumption \eqref{ABCfunAssm},
a function $f:\ZZ\sp{n} \to \Rinf$ 
is M-convex (resp., \Mnat-convex) 
if and only if,
for any vector $p \in \RR\sp{n}$, $\argmin f[-p]$ is an M-convex (resp., \Mnat-convex) set
or an empty set.
\finbox
\end{theorem}

\subsubsection{\MM-convex functions}
\label{SCm2fn}

A nonempty set $S$ $\subseteq \ZZ\sp{n}$
is called {\em  \MM-convex}  
if it can be represented as the intersection of two {\rm M}-convex sets
\cite[Section~4.6]{Mdcasiam}.
Similarly, a nonempty set $S$ $\subseteq \ZZ\sp{n}$
is called {\em \MMnat-convex}  
if it can be represented as the intersection of two \Mnat-convex sets
\cite[Section~4.7]{Mdcasiam}.
An \MM-convex set $S$ lies on a hyperplane with a constant component sum
(i.e., $x(N) = y(N)$ for all  $x, y \in S$).

A function 
$f: \ZZ\sp{n} \to \RR \cup \{ +\infty \}$
is said to be 
{\em \MM-convex}
if it can be represented as the sum of
two M-convex functions
$f_{1}$ and $f_{2}$, that is,
if
\[
f(x) =  f_{1}(x) + f_{2}(x)
\qquad (x \in \ZZ\sp{n}).
\]
The sum of two \Mnat-convex functions
is called an {\em \MMnat-convex function}.
A nonempty set $S$ is \MM-convex (resp., \MMnat-convex)
if and only if its indicator function
$\delta_{S}$ is \MM-convex (resp., \MMnat-convex).

\medskip

\noindent {\bf Acknowledgement}. 
This work was supported by JSPS KAKENHI Grant Numbers 
JP17K00037, JP20K11697, and JP21K04533.
The authors thank Akihisa Tamura for careful reading of the manuscript.





\newpage
\tableofcontents

\end{document}